\documentclass[preprint,12pt]{elsarticle}
\usepackage{}
\usepackage{amssymb}
\usepackage{amsfonts}
\usepackage{mathrsfs}
\usepackage{amsmath}
\usepackage{amscd}
\usepackage{bbm}
\usepackage{amssymb}
\usepackage{amsmath}
\usepackage{graphicx}
\usepackage{array}
\usepackage{bm}
\usepackage{color}
\usepackage[colorlinks=true]{hyperref}
\usepackage[titletoc]{appendix}

\journal{}

\newtheorem{prop}{Proposition}[section]
\newtheorem{theorem}{Theorem}[section]

\newtheorem{lemma}{Lemma}[section]

\newtheorem{remark}{Remark}[section]

\renewcommand{\appendixname}{A}
\numberwithin{equation}{section}

\begin{document}
\begin{frontmatter}

\title{$L^p$ estimates for fractional Schr\"{o}dinger operators with Kato class potentials}

\author{$^a$ Shanlin Huang} \ead{shanlin\_huang@hust.edu.cn}

\author{$^b$ Ming Wang \corref{cor1}}\cortext[cor1]{Corresponding author.} \ead{mwangcug@outlook.com, Tel.: +86 027 67883091}

\author{$^a$ Quan Zheng} \ead{qzheng@hust.edu.cn}

\author{$^a$ Zhiwen Duan} \ead{duanzhw@hust.edu.cn}

\address{$^a$ School of Mathematics and Statistics, Huazhong University of Science and Technology,  Wuhan,  430074, China.}

\address{$^b$ School of Mathematics and Physics, China University of
Geosciences, Wuhan, Hubei, 430074,   China.}



\begin{abstract}
Let $\alpha>0$, $H=(-\Delta)^{\alpha}+V(x)$, $V(x)$ belongs to the higher order Kato class $K_{2\alpha}(\mathbb{R}^n)$.  For $1\leq p\leq \infty$, we prove a polynomial upper bound of $\|e^{-itH}(H+M)^{-\beta}\|_{L^p, L^p}$ in terms of time $t$ for all integers $\alpha$ and $2\alpha\ge[\frac n2]+1$ if $\alpha$ is not an integer. Both the smoothing exponent $\beta$ and the growth order in $t$  are almost optimal compared to the free case. The main ingredients in our proof are pointwise heat kernel estimates for the semigroup $e^{-tH}$. We obtain a Gaussian upper bound with sharp coefficient for integral $\alpha$ and a polynomial decay for fractional $\alpha$.
\end{abstract}

\end{frontmatter}

\section{Introduction}
Let $\alpha$ be a positive number. Consider the fractional Schr\"{o}dinger equation
\begin{equation}\label{equ1.01}
i\partial_{t}u +(-\Delta)^{\alpha}u+V(x)u=0, \quad (t,x)\in \mathbb{R}\times \mathbb{R}^n,
\end{equation}
with initial condition
\begin{equation}\label{equ1.02}
u(0,x)=u_{0}(x).
\end{equation}
The function $V(x)$ is often called {the} potential, the fractional Laplacian $(-\Delta)^\alpha$ is defined as a Fourier multiplier with symbol $|\xi|^{2\alpha}$. For $\alpha\in (0,1)$, the equation \eqref{equ1.01}-\eqref{equ1.02} was introduced by Laskin \cite{Laskin} as a result of extending the Feynman path integral, from the Brownian-like to Levy-like quantum mechanical paths. Recently, the fractional Schr\"{o}dinger equation has been studied extensively in literatures, since it appears naturally in problems involving
nonlinear optics, plasma physics and other areas. We just refer the readers to \cite{DZF, HYZ, kit} for spectral and scattering theories and to \cite{Cho, guo, hong} for well posedness and ill posedness of nonlinear equations, respectively.


In this paper, we are interested in  {investigating} the $L^p$ estimates for solutions  {of the equation} \eqref{equ1.01}-\eqref{equ1.02} with potentials $V$ belonging to the higher order Kato class $K_{2\alpha}(\mathbb{R}^n)$. Recall that, for every $\alpha>0$, a real-valued measurable function $V(x)$ on $\mathbb{R}^n$ is said to lie in $K_{2\alpha}(\mathbb{R}^n)$ if
$$
\sup_{x\in\mathbb{R}^n}\int_{|x-y|<1}{|V(y)|dy}< \infty, \qquad \textmd{for}~ 2\alpha> n,
$$
and
$$
\lim_{\delta \to 0}\sup_{x\in\mathbb{R}^n}\int_{|x-y|<\delta}{\omega_{\alpha}(x-y)|V(y)|dy}=0, \qquad \textmd{for}~ 0<2\alpha\leq n,
$$
where
\begin{equation*}
\omega_{\alpha}(x) =
\left\{\begin{array}{cl}
|x|^{2\alpha-n}, & \textmd{if }~ 0<2\alpha<n,\\
\ln{|x|^{-n}}, & \textmd{if }~ 2\alpha=n.
\end{array}\right.
\end{equation*}
Note that $K_{2}(\mathbb{R}^n)$ coincides with the classical Kato class. If $p\neq 2$ and $t\neq0$, it is well known that the free Schr\"{o}dinger group $e^{it\triangle}$ is not bounded on $L^p(\mathbb{R}^n)$(see e.g. \cite{Hormander}), thus the solution $u(t)$ of \eqref{equ1.01}-\eqref{equ1.02} is not necessarily  bounded on $L^p(\mathbb{R}^n)$ for general $u_0 \in L^p(\mathbb{R}^n)$. Nevertheless, we shall establish quantitative bounds of the solution for {equations} \eqref{equ1.01}-\eqref{equ1.02} on $L^p(\mathbb{R}^n)$ if some smoothness assumptions are imposed on the initial data. Here is our main result.

\begin{theorem}\label{thm1.2}
Let $\alpha$ be a positive integer or {let} $2\alpha\ge[\frac n2]+1$ if $\alpha$ is not an integer. {Assume that} $H=(-\Delta)^{\alpha}+V(x)$, {where} $V\in K_{2\alpha}(\mathbb{R}^n)$.  Then, for $M$ large enough, there exists a positive constant $C$ independent of $t$ so that the estimate
\begin{equation}\label{1.02}
\|e^{-itH}(H+M)^{-\beta}\|_{L^p, L^p}\leq C(1+|t|)^{\gamma},\,\, t\in \mathbb{R}
\end{equation}
holds for all $1\leq p\leq \infty$ and $\beta,\gamma>n|\frac{1}{2}-\frac{1}{p}|$.
\end{theorem}

Some related works are summarized as follows:
\begin{center}
\begin{tabular}{|c|c|c|c|c|}
\hline
Brenner, Thom$\acute{e}$e $\&$Wahlbin \cite{BTW} &$\alpha = 1$ & $V=0$ \\ \hline
 Fefferman $\&$ Stein \cite{Fefferman} & $\alpha>0$ & $V=0$  \\ \hline
Jensen $\&$ Nakamura \cite{A.Jensen1, A.Jensen2} & $\alpha=1$ & $V\in K_2(\mathbb{R}^n)$  \\ \hline
\end{tabular}
\end{center}
Clearly, these works are contained in \ref{thm1.2}.
Moreover, both the smoothing exponent $\beta$ and the growth order $\gamma$ are {almost} optimal in the sense that they can not be improved for the free Schr\"{o}dinger   evolution $e^{it\triangle}$ \cite{BTW}.

Now we turn to the proof of Theorem \ref{thm1.2}.
There is a {powerful} approach to deduce \eqref{1.02} by heat kernel estimates. In fact, let $\{e^{-zH}\}_{z\in \mathbb{C}^{+}}$, $\mathbb{C}^{+}$ is the open
right half-plane, be an analytic bounded semigroup on $L^2(\mathbb{R}^n)$, and let $K$ be the integral kernel associated with {$e^{-tH}$}, namely
$$
(e^{-tH}f)(x)=\int_{\mathbb{R}^n} K(t,x,y)f(y)dy, \quad f \in L^2(\mathbb{R}^n).
$$
If the estimate
\begin{equation}\label{1.03}
|K(t,x,y)|\leq C_{1}t^{-\frac{n}{2m}}\exp\left\{-C_{2}\frac{|x-y|^{\frac{2m}{2m-1}}}{t^{\frac{1}{2m-1}}}+C_3t\right\},\,~~t>0,
\end{equation}
holds for some positive constants $C_1, C_2, C_3$, then \eqref{1.02} follows. The result is proved partially by Zheng and Zhang \cite{zheng99}, and fully by Carron, Coulhon and Ouhabaz \cite{CCO}. We also note that recently, {the work by D'Ancona and Nicola \cite{DN} (see also \cite{BDN}) has made a far-reaching generalization in the abstract level.} They can obtain sharp $L^p$ estimates by some general assumptions on $e^{-tH}$. Thus, the proof of Theorem \ref{thm1.2} is reduced to establishing the kernel estimates of the semigroup $e^{-t((-\Delta)^\alpha+V)}$, which are the main contributions of the paper. {The asymptotic behavior of the heat kernel, however, turns out to be quite different regarding whether $\alpha$ is an integer or not.} So we divide our {discussions} into two {cases, which presented in Section \ref{sec2} and \ref{sec3} respectively.}

Section \ref{sec2} is devoted to the case when $\alpha=m$ is a positive integer. We shall prove the following quantitative bounds of the heat kernel.
\begin{theorem}\label{thm1.3}
Let $V \in K_{2m}(\mathbb{R}^n)$. Then there are constants $\varepsilon_1\in(0,1)$ and $C>0$ depending only on $m,n$  such that the kernel of the semigroup $e^{-t((-\Delta)^m + V)}$ satisfies
\begin{align*}
|K(t,x,y)| \leq C \varepsilon^{-(n+1)([\frac{n}{2m}]+3)}t^{-\frac{n}{2m}}\exp\left\{-\varsigma_m\frac{|x - y|^{\frac{2m}{2m - 1}}}
{(1+\varepsilon)t^{\frac{1}{2m - 1}}} +\varepsilon^{-(n+3)} V_{\varepsilon^{n+3}} t \right\}
\end{align*}
for all $t>0, \varepsilon\in (0,\varepsilon_1)$, where $[s]$ denotes the integer part of $s$, $V_{\varepsilon^2}$ depending on $V$ and  $\varepsilon$ (see Section 2 for the definition), $\varsigma_m = (2m - 1)(2m)^{-\frac{2m}{2m - 1}}\sin\left(\frac{\pi}{4m - 2}\right)$.
\end{theorem}

A few remarks are given in order:

\textbf{(1).} If $m=1$, Theorem \ref{thm1.3} follows easily from the famous Feynman-Kac path formula \cite{Simon}. However, it is known that the F-K formula is not available for higher order Schr\"{o}dinger semigroup \cite{E. B. Davies2}. Zheng and Yao \cite{zheng09} tried to overcome this difficulty by a purely semigroup approach, but only obtained $L^p-L^q$ estimates.

\textbf{(2).} Recently, Deng, Ding, and Yao \cite{DDY} established \eqref{1.03} for the kernel of $e^{-t((-\Delta)^m + V)}$ using Davies' perturbation method\cite{E. B. Davies95}. Compared to \cite{DDY}, Theorem \ref{thm1.3} gives quantitative information on the constants appearing in the {estimate} \eqref{1.03}. Moreover, the constant $\varsigma_m$ is sharp according to the work of Barbatis and Davies \cite{Barbatis96}, and Dungey  \cite{Dungey}. We mention that some heat kernel estimates with sharp constant $\varsigma_m$ of elliptic operators of order $2m$ are obtained by Barbatis \cite{Bar98, Bar01, G. Barbatis4}, however the restriction $2m>n$ is needed there for technical reasons.

\textbf{(3).} In order to obtain the sharp constant $\varsigma_m$, we consider the conjugated operator $P_\lambda(D) = e^{-\lambda \phi}(-\Delta)^{m}e^{\lambda \phi}$ with $\phi$ being linear, inspired by the work of Barbatis and Davies \cite{Barbatis96}. This differs from the treatment in \cite{DDY}, where $\phi$ is chosen to be a bounded smoothing function. The choice of linear weight  allows us to obtain more precise resolvent estimates in $L^1(\mathbb{R}^n)$, by taking full advantage of pointwise kernel bounds of $e^{-z(-\Delta)^{m}}$ with ``complex time'' $\Re z> 0$.

\textbf{(4).} Compared to \cite{ Bar98, Bar01, G. Barbatis4, Barbatis96}, the main difficulty lies in that the sum $P_\lambda(D) + V$ can not be understood in the operator sense on $L^2(\mathbb{R}^n)$ for general $V \in K_{2m}(\mathbb{R}^n)$. To overcome this difficulty,  we construct a sesquilinear form $Q_{\lambda,V}$, and show the associated operator generates an analytic semigroup on $L^2(\mathbb{R}^n)$. After that, we further check that the semigroup constructed on $L^1(\mathbb{R}^n)$ and  on $L^2(\mathbb{R}^n)$ are consistent. Finally, we prove a quantitative $L^1- L^\infty$ type estimate of the heat semigroup $e^{-t(P_\lambda(D) + V)}$ by its smoothing effect, {which implies} Theorem \ref{thm1.3} via a standard method.

Section 3 is devoted to the kernel estimates for $e^{-t((-\Delta)^{\alpha}+V)}$ with fractional $\alpha>0$. Our result is contained in the following

\begin{theorem}\label{thm1.4}
Assume that $V \in K_{2\alpha}(\mathbb{R}^n)$, $\alpha\in \mathbb{R}_{+}\setminus\mathbb{N}$. Then for every $0 < \varepsilon \ll 1$, there are constants $C$, and $\mu_{\varepsilon, V}$ depending on $V$ and  $\varepsilon$, such that  the kernel of the semigroup $e^{-t((-\Delta)^{\alpha} + V)}$ satisfies
\begin{align}\label{equ1.08}
|K(t,x,y)|\leq Ce^{\mu_{\varepsilon, V} t}\frac{t}{(|x-y|^2+t^{\frac{1}{\alpha}})^{\frac n2+\alpha}}, \quad t>0.
\end{align}
\end{theorem}
The decay rate $\frac{n}{2}+\alpha$ in \eqref{equ1.08} is sharp in the sense that it can not be improved in the free case \cite{bal,Miyaj}. In the case $0<\alpha<1$, Theorem \ref{thm1.4} is contained in \cite[Proposition 4.1]{shi}, though the fact is not pointed out explicitly there. The proof in \cite{shi} is based on Voigt's theory of absorption semigroup \cite{voi86, voi88}. However, it only works for positivity  semigroups and thus fails in the case $\alpha>1$.  In \cite{Bogdan.K}, Bogdan and Jackubowski  proved the estimate \eqref{equ1.08} for the kernel of the semigroup generated by $(-\Delta)^{\alpha}+b(x)\cdot\nabla$ for $\frac{1}{2}<\alpha<1$ and $V\in K_1(\mathbb{R}^n)$. Xie and Zhang \cite{zhang} improved the result to $\alpha=\frac{1}{2}$. The main tools in \cite{Bogdan.K}  are the so-called the $3P$ Theorem and a characterization of Kato class potentials $K_{2\alpha}(\mathbb{R}^n)$. In Section 3, we adapt these tools to the case $\alpha>1$ and prove Theorem \ref{thm1.4}. Our new ingredient in this part is to show that $-((-\Delta)^{\alpha}+V)$ is the generator of the constructed semigroup in the usual sense, while only weak generator was obtained in \cite{Bogdan.K, zhang}.  To this end, we shall exploit the analyticity of the semigroup generated by $-(-\Delta)^\alpha$ and the uniqueness of vector-valued Laplace transform.

 {In section 4, we shall apply our heat kernel estimates to establish Theorem \ref{thm1.2}.} Note that when $\alpha=m$ is an positive integer, then Theorem \ref{thm1.2} follows from Theorem \ref{thm1.3} directly as mentioned before. However, when $\alpha$ is not an integer, the approach in \cite{CCO} does not work since Theorem \ref{thm1.4} gives only polynomial decay in $|x-y|$ of the kernel as $|x-y|$ goes to infinity. We mention that D'Ancona and Nicola's interesting result \cite[Section 5]{DN} indicates that if the kernel of $e^{-tH}$ satisfies \eqref{equ1.08}, then the $L^p$ estimate \eqref{1.02} is valid for $2\alpha>[\frac n2]+1$. In our particular case $H=(-\Delta)^{\alpha}+V(x)$, we can further show that the result is also true when $2\alpha=[\frac n2]+1$ by adapting methods in  \cite{DN, A.Jensen1, A.Jensen2}.

\section{Sharp heat kernel estimates for the integer case}\label{sec2}

In this section, we always assume that $m$ is a positive integer and the potential $V$ belongs to $K_{2m}(\mathbb{R}^n)$.
For any $\varepsilon>0$, it follows from \cite{zheng09} that, there exists  a constant $\sigma_\varepsilon>0$ such that
\begin{equation}\label{equ-L1-small}
\|V\varphi\|_{L^1} \leq \varepsilon \|(-\Delta)^{m} \varphi\|_{L^1} + \sigma_\varepsilon \|\varphi\|_{L^1}, \quad \textmd{for all } \varphi \in \mathcal {L}^{2m,1}(\mathbb{R}^n),
\end{equation}
where $\mathcal {L}^{2m,1}(\mathbb{R}^n)$ denotes the $2m$ order Bessel space in $L^1(\mathbb{R}^n)$.

Fixed $\varepsilon \in (0,1)$,  we define a number
$$
V_\varepsilon = \inf{}\{\sigma: \sigma \in E_\varepsilon\}
$$
where
$$
E_\varepsilon = \bigg\{\sigma \geq 0: \|V\varphi\|_{L^1} \leq \varepsilon \|(-\Delta)^{m} \varphi\|_{L^1} + \sigma \|\varphi\|_{L^1}, \quad \textmd{for all } \varphi \in \mathcal {L}^{2m,1}(\mathbb{R}^n)\bigg\}.
$$
According to \eqref{equ-L1-small}, it is easy to see that $E_\varepsilon$ is a non-empty, bounded from below, connected set, thus $V_\varepsilon$ is well defined.

Let $a \in \mathbb{R}^n$, $|a| = 1$, $\lambda > 0$. Consider the following operator
$$
P_{\lambda}(D) = e^{-\lambda a\cdot x}(-\Delta)^{m}e^{\lambda a\cdot x}
$$
with maximal domain in $L^1(\mathbb{R}^n)$. This is a partial differential operator with constant coefficients  though the spacial variable $x$ get involved in the expression. In fact, it is easy to check that
$$
P_{\lambda}(D) = (-\Delta - \lambda^2 + 2i\lambda a\cdot D)^m
$$
with $D = (\frac{1}{i}\frac{\partial}{\partial x_1}, \cdots, \frac{1}{i}\frac{\partial}{\partial x_n})$, and its symbol $P_\lambda(\xi)=(|\xi|^2-\lambda^2+2i\lambda a\cdot\xi)^m$ satisfies the following property (see \cite{Barbatis96})
\begin{equation}\label{equ2.1}
\Re  P_\lambda(\xi)\geq - b_m \lambda^{2m}, \quad \textmd{ for all}~ \xi \in \mathbb{R}^n,
\end{equation}
where $b_m$ = ($\sin\frac{\pi}{4m - 2})^{-(2m - 1)}$. Here and below, we use $\Re f (\mathfrak{I}(f))$ to denote  the real (imaginary) part of $f$.

Denote by $H_{\lambda, V} = P_\lambda(D) + V$, $P(D) = (-\Delta)^m$ with domain $\mathcal {L}^{2m,1}(\mathbb{R}^n)$.

\begin{theorem}\label{thm1}
Let $V\in K_{2m}(\mathbb{R}^n)$. Then the operator $- H_{\lambda, V}$ generates an analytic semigroup on $L^1(\mathbb{R}^n)$. Moreover, the estimate
\begin{equation*}
\|e^{-t H_{\lambda, V}}\|_{L^1,L^1}+\|t H_{\lambda, V}e^{-t H_{\lambda, V}}\|_{L^1,L^1}\leq C\varepsilon^{-(n+1)}e^{t\iota(\varepsilon, \lambda)},
\end{equation*}
holds for all $t > 0$, $a \in \mathbb{R}^n$, $|a| = 1$, $\lambda > 0$, $0<\varepsilon<\varepsilon_1$,  where $\iota(\varepsilon, \lambda):= \varepsilon^{-(n+3)}V_{\varepsilon^{n+3}} +  (b_m+\varepsilon)\lambda^{2m}$, $\varepsilon_1$ and $C$ are some small and large enough constants depending only on $m$ and $n$, respectively.
\end{theorem}
\begin{proof}
We divide the proof into three steps.

\textbf{Step 1.} The estimates of $(\mu + P_{\lambda}(D))^{-1}$. The resolvent can be understood as an integral operator. More precisely
$$
(\mu + P_{\lambda}(D))^{-1}\varphi = h(\cdot, \mu) * \varphi(\cdot), \quad \varphi \in L^1(\mathbb{R}^n),
$$
where $h(x, \mu) = \int_0^\infty \int_{\mathbb {R}^n} e^{ix\cdot \xi}e^{ -z(\mu + P_{\lambda}(\xi))}d\xi dz$, $*$ denotes the convolution. To show the boundedness of $(\mu + P_{\lambda}(D))^{-1}$ on $L^1(\mathbb{R}^n)$, we need to establish some bounds of $h(x,\mu)$.

To this end, we recall the following kernel estimate of $e^{-t(-\triangle)^m}$ in \cite{Barbatis96, Dungey},
\begin{equation}\label{equ2.2}
\left|\int_{\mathbb {R}^n} e^{ix\cdot \xi} e^{- t|\xi|^{2m}}d\xi \right| \leq Ct^{-n/2m}\exp\left\{-\varsigma_m|x|^{\frac{2m}{2m-1}}/t^{\frac{1}{2m-1}}\right\}, \quad t>0,
\end{equation}
where $\varsigma_m = (2m - 1)(2m)^{-\frac{2m}{2m - 1}}\sin\left(\frac{\pi}{4m - 2}\right)$.
From this, similar to (\cite{zheng99}, Theorem 2.1), we have
\begin{equation}\label{equ2.3}
\left|\int_{\mathbb {R}^n} e^{ix\cdot \xi} e^{- z |\xi|^{2m}}d\xi \right| \leq C|\Re  z|^{-n/2m}, \quad \Re z > 0.
\end{equation}
Using \eqref{equ2.2} and \eqref{equ2.3}, repeating word by word of the proof of Lemma 9 in \cite{E. B. Davies95}, we obtain
\begin{equation}\label{equ2.4}
\left|\int_{\mathbb {R}^n} e^{ix\cdot \xi} e^{- z |\xi|^{2m}}d\xi \right| \leq C|z|^{-n/2m}\exp\left\{-(1-\varepsilon)\varsigma_m|x|^{\frac{2m}{2m-1}}/|z|^{\frac{1}{2m-1}}\right\}
\end{equation}
for $\Re z > 0, |\arg z| \leq \theta, \theta = \arctan (\varepsilon/C), 0<\varepsilon<1$.

Changing variable $\xi \rightarrow \xi - i\lambda a$, and shifting the path of integration in the definition of $h$ yields that
\begin{equation}\label{equ-h-complex}
h(x, \mu) = e^{\lambda a \cdot x}\int_0^\infty \int_{\mathbb {R}^n} e^{ix\cdot \xi} \exp\left\{- ( \mu+ |\xi|^{2m})\rho e^{i\theta~ sgn(\Im \mu) }\right\}d\xi d\rho,
\end{equation}
where $sgn (\cdot)$ denotes the standard sign function, $\theta$ is the same as that in \eqref{equ2.4}.
Thus, using \eqref{equ2.4},  we deduce  that for $\rho>0$
\begin{multline}\label{equ-h-1}
\left|\int_{\mathbb {R}^n} e^{ix\cdot \xi} \exp\left\{- ( \mu+ |\xi|^{2m})\rho e^{i\theta~ sgn(-\Im \mu) }\right\}d\xi\right|\\
\leq Ce^{-(\Re \mu \rho\cos\theta+|\Im\mu|\rho\sin\theta)}\rho^{-\frac{n}{2m}} \exp\left\{-(1-\varepsilon)\varsigma_m|x|^{\frac{2m}{2m-1}}/\rho^{\frac{1}{2m-1}}\right\}.
\end{multline}
Using Young inequality we find for all $\lambda>0, x\in \mathbb {R}^n$ and $\varepsilon\in (0, \frac{1}{2})$,
\begin{align}\label{equ-h-2}
e^{\lambda a \cdot x}\leq \exp\{(1-2\varepsilon)\varsigma_m|x|^{\frac{2m}{2m-1}}+(1+C'\varepsilon)b_m\lambda^{2m}\}
\end{align}
for all $\varepsilon\in(0,\frac{1}{4})$ and some constant $C'$ depending only on $m$.

It follows from \eqref{equ-h-complex}-\eqref{equ-h-2} that
\begin{align*}
|h(x,\mu)|\leq C(\Re \mu \cos\theta - (1+c'\varepsilon)b_m\lambda^{2m}+|\Im\mu|\sin\theta)^{-1}e^{-\varepsilon\varsigma_m|x|^{\frac{2m}{2m-1}}}
\end{align*}
for all  $\Re \mu>(1+c'\varepsilon)b_m\lambda^{2m}$ and $\varepsilon\in(0,\frac{1}{4})$. Hence
\begin{align*}
&\|(\mu + P_{\lambda}(D))^{-1}\|_{L^1,L^1} \leq \|h(\cdot, \mu)\|_{L^1} \\
&\leq C\varepsilon^{-\frac{n(2m-1)}{2m}}(\Re \mu \cos\theta - (1+C'\varepsilon)b_m\lambda^{2m}+|\Im\mu|\sin\theta)^{-1}\\
&\leq C\varepsilon^{-n}(\sin\theta)^{-1}/|\Re(\mu-(\cos\theta)^{-1}(1+C'\varepsilon)b_m\lambda^{2m})(\tan\theta)^{-1}+|\Im\mu||\\
&\leq C\varepsilon^{-(n+1)}/|\mu - (b_m+C''\varepsilon)\lambda^{2m}|,
\end{align*}
for all $\Re \mu > (b_m+C''\varepsilon)\lambda^{2m}$, $0<\varepsilon<c_1$, where $c_1(C'')$ is a small (big) a constant depending only on $m,n$. Here we used the fact that $\theta=\arctan(\varepsilon/C)$, $\cos \theta \sim 1$, $\sin\theta\sim \tan\theta \sim \varepsilon$ when $\varepsilon\rightarrow 0$. A scaling argument yields that
\begin{equation}\label{equ2.5}
\|(\mu + P_{\lambda}(D))^{-1}\|_{L^1,L^1}  \leq C\varepsilon^{-(n+1)}/|\mu - (b_m+\frac{\varepsilon}{2})\lambda^{2m}|
\end{equation}
for all $ \Re \mu > (b_m+\frac{\varepsilon}{2})\lambda^{2m}$, $0<\varepsilon<c_1$ with a different $C$.

\textbf{Step 2.} The estimates of $(\mu + P_{\lambda}(D) + V)^{-1}$. Clearly, we have
$$
(\mu + P_{\lambda}(D) + V)^{-1} = (\mu + P_{\lambda}(D))^{-1}(I + V(\mu + P_\lambda(D))^{-1})^{-1}
$$
and
$$
V(\mu + P_\lambda(D))^{-1} = V(\mu + P(D))^{-1}[I+(P(D)-P_\lambda(D))(\mu + P_\lambda(D))^{-1}].
$$
We claim that there exists a constant $C$ depending only on $m,n$ such that for all $\varphi\in \mathcal {L}^{2m,1}(\mathbb{R}^n), \lambda>0$
\begin{equation}\label{equ6}
\|(P(D)-P_\lambda(D))\varphi\|_{L^1} \leq C(\|P_\lambda(D)\varphi\|_{L^1}+\lambda^{2m}\|\varphi\|_{L^1}).
\end{equation}
In fact, fix a multi-index $\alpha$ satisfying $|\alpha| \leq 2m-1$. According to the Bernstein theorem (see Proposition \ref{lemA1}), we find that $D^\alpha(1+(-\triangle)^{m})^{-1}$ is bounded on $L^1(\mathbb{R}^n)$. From this and \eqref{equA2},  we have
\begin{equation}\label{equ7}
r^{|\alpha|}\|D^\alpha\varphi\|_{L^1} \leq C(r^{2m}\|(-\Delta)^m\varphi\|_{L^1} + \|\varphi\|_{L^1}), \quad\varphi \in \mathcal {L}^{2m,1}(\mathbb{R}^n)
\end{equation}
for all $r>0$. For any $\delta \in (0,1)$, we can choose $r = (\delta \lambda^{-q})^{\frac{1}{2m-|\alpha|}}$ with $q=2m-|\alpha|$ in \eqref{equ7} to get
\begin{eqnarray}\label{equ-inter}
\|\lambda^qD^\alpha\varphi\|_{L^1} &\leq& C(\delta\|(-\Delta)^m\varphi\|_{L^1} + \delta^{-\frac{|\alpha|}{2m-|\alpha|}}\lambda^{2m}\|\varphi\|_{L^1}).
\end{eqnarray}
 Since $P_{\lambda}(D)$ can be rewritten as
$$
P_{\lambda}(D)\varphi = (-\Delta)^m\varphi + \sum_{|\beta|<2m, ~q=2m-|\beta|}c_{\beta}\lambda^qD^\beta\varphi,
$$
we apply \eqref{equ-inter} to each term in the sum and choose $\delta$ to be small to obtain
$$
\|P_{\lambda}(D)\varphi-P(D)\varphi\|_{L^1} \leq \frac{1}{2}\|P(D)\varphi\|_{L^1} + C\lambda^{2m}\|\varphi\|_{L^1},
$$
which implies \eqref{equ6}.

It follows from \eqref{equ6} and \eqref{equ2.5} that
\begin{align}\label{equ-L1-step2-4}
\|(P(D)-P_\lambda(D))(\mu + P_\lambda(D))^{-1}\|_{L^1,L^1} &\leq C+ C(|\mu|+\lambda^{2m})\|(\mu+P_\lambda(D))^{-1}\|_{L^1,L^1}\nonumber\\
&\leq C\varepsilon^{-(n+2)}, \quad 0<\varepsilon<c_1.
\end{align}

Thanks to Proposition \ref{lemA2}, we have
\begin{align}\label{equ-L1-step2-5}
\|(\mu+P(D))^{-1}\|_{L^1,L^1}\leq C/|\mu|, \quad \Re \mu>0.
\end{align}
 Rewrite $P(D)(\mu+P(D))^{-1} = I-\mu(\mu+P(D))^{-1}$,  by \eqref{equ-L1-step2-5}, we obtain
\begin{align}\label{equ-L1-step2-6}
\|P(D)(\mu+P(D))^{-1}\|_{L^1,L^1}\leq C, \quad \Re \mu>0.
\end{align}

By the definition of $V_\varepsilon$, using \eqref{equ-L1-step2-5} and \eqref{equ-L1-step2-6} we get
\begin{align}\label{equ-L1-step2-7}
\|V(\mu + P(D))^{-1}\|_{L^1,L^1} &\leq \varepsilon^{n+3}\|P(D)(\mu+P(D))\|_{L^1,L^1} + V_{\varepsilon^{n+3}}\|(\mu+P(D))^{-1}\|_{L^1,L^1}\nonumber\\
& \leq C\varepsilon^{n+3}
\end{align}
for $\Re \mu > \varepsilon^{-(n+3)}V_{\varepsilon^{n+3}}$.

It follows from \eqref{equ-L1-step2-4} and \eqref{equ-L1-step2-7} that
\begin{eqnarray*}
\lefteqn{\|V(\mu + P_\lambda(D))^{-1} \|_{L^1,L^1} }\\&\leq& \|V(\mu + P(D))^{-1}\|_{L^1,L^1}(1+\|(P(D)-P_\lambda(D))(\mu + P_\lambda(D))^{-1}\|_{L^1,L^1})\\
&\leq& C\varepsilon \leq 1/2
\end{eqnarray*}
for $\Re \mu > \iota(\varepsilon, \lambda):= \varepsilon^{-(n+3)}V_{\varepsilon^{n+3}} +  (b_m+\varepsilon)\lambda^{2m}$ and $0<\varepsilon<c_2 =: \min \{c_1, 1/2C\} $. This implies that $I + V(\mu + P_\lambda(D))$ is invertible on $L^1(\mathbb{R}^n)$, and the norm of the inverse is bounded by $2$.

Therefore,
\begin{align}\label{equ-L1-step2-8}
\|(\mu + H_{\lambda, V})^{-1}\|_{L^1,L^1} \leq C\varepsilon^{-(n+1)}/|\Im \mu|
\end{align}
for all $\Re \mu >  \iota(\varepsilon, \lambda), \Im \mu \neq 0$, $0<\varepsilon<c_2$.

\textbf{Step 3.} The estimates of the semigroup. By the resolvent estimates \eqref{equ-L1-step2-8} and the standard method, see e.g. \cite[pp.61--63]{Pazy}, it can be shown that
the resolvent set
$$
\rho(-H_{\lambda, V}) \supset \left\{z: |arg (z-  \iota(\varepsilon, \lambda))| < \frac{\pi}{2} + \delta \right\} \cup \{0\}
$$
where $\delta = 2^{-1}\arctan(C^{-1}\varepsilon^{n+1})$, and in this region
$$
\|(\mu + H_{\lambda, V})^{-1}\|_{L^1,L^1}\leq C\varepsilon^{-(n+1)}/|\mu-  \iota(\varepsilon, \lambda)|.
$$
Thus, $ -H_{\lambda, V} - \iota(\varepsilon, \lambda) $ generates an analytic semigroup in $L^1(\mathbb{R}^n)$, and
$$
\|e^{-t(H_{\lambda, V} + \iota(\varepsilon, \lambda))}\|_{L^1,L^1} + \|t(H_{\lambda, V}+\iota(\varepsilon, \lambda))e^{-t(H_{\lambda, V} + \iota(\varepsilon, \lambda))}\|_{L^1,L^1} \leq C\varepsilon^{-(n+1)}, \quad t>0.
$$
This implies the desired conclusion obviously.
\end{proof}

From Theorem \ref{thm1}, we know that the operator sum $-H_{\lambda, V}$ generates an analytic semigroup on $L^1(\mathbb{R}^n)$. For our purpose, we shall extend the semigroup on $L^p(\mathbb{R}^n)(1\leq p < \infty)$. To this end, we need to show $-H_{\lambda, V}$ also generates an analytic semigroup on $L^2(\mathbb{R}^n)$ in some sense. Note that $H_{\lambda, V}$ may make no sense as an operator acting on $L^2(\mathbb{R}^n)$ for general $V\in K_{2m}(\mathbb{R}^n)$, thus we shall study $H_{\lambda, V}$ on $L^2(\mathbb{R}^n)$ in the framework of sesquilinear forms in Hilbert spaces developed by  Kato \cite{Kato}.

Let $V \in K_{2m}$ and let $a \in \mathbb{R}^n$, $|a| = 1$, $\lambda > 0$.  Define a sesquilinear form
$$
 Q_{\lambda,V}[\varphi, \psi] =
\left(e^{-\lambda a \cdot x}(-\Delta)^{\frac{m}{2}}e^{\lambda a \cdot x}\varphi, e^{\lambda a \cdot x}(-\Delta)^{\frac{m}{2}}e^{-\lambda a \cdot x}\psi\right) + \int V\varphi\overline\psi dx,
$$
where $\varphi, \psi \in D( Q_{\lambda,V}) = H^m(\mathbb{R}^n)$, $(\cdot,\cdot)$ means the inner product in $L^2(\mathbb{R}^n)$.  Note that the form is well defined since $V \in K_{2m}$, see \cite[Theorem 4.2]{zheng09}. In what follows, we denote by $Q_{\lambda,V}[\varphi] = Q_{\lambda,V}[\varphi, \varphi]$, $Q_{\lambda} = Q_{\lambda,0}, Q_0 = Q_{0,0}$ for brevity.

\begin{theorem}\label{thm2}
The form $ Q_{\lambda,V}$ is a closed sectorial form, thus it is associated with a unique m-sectorial operator $\widetilde{H}_{\lambda, V}$ in the sense that
$$
Q_{\lambda,V}(\varphi, \psi) = (\widetilde{H}_{\lambda, V}\varphi, \psi), \quad \varphi \in D(\widetilde{H}_{\lambda, V}), \psi \in H^m(\mathbb{R}^n).
$$
Moreover, $-\widetilde{H}_{\lambda, V}$ generates an analytic semigroup satisfying
\begin{multline*}
\left\|e^{-t\widetilde{H}_{\lambda, V}}\right\|_{L^2,L^2} + \left\|t\widetilde{H}_{\lambda, V}e^{-t\widetilde{H}_{\lambda, V}}\right\|_{L^2,L^2}\\ \leq C\varepsilon^{-2m}\exp\left\{(b_m + \varepsilon)\lambda^{2m}t + C\varepsilon^{-(2m+2)}V_{\varepsilon^{4m+4}}t\right\}
\end{multline*}
for all $\varepsilon \in (0, \varepsilon_2),  t>0$, where $\varepsilon_2$ and $C$ are constants depending only on $m,n$.
\end{theorem}
\begin{proof}
At first, we shall investigate the numerical range of $Q_{\lambda, V}$ defined by
$$
\Theta(Q_{\lambda, V}):= \{Q_{\lambda, V}[\varphi]: \varphi\in H^m(\mathbb{R}^n), \|\varphi\|_{L^2}=1\}.
$$
To this end, for $\varepsilon \in (0,1)$, we rewrite the form $Q_{\lambda, V}[\varphi]$ as
\begin{equation}\label{equ8}
Q_{\lambda, V}[\varphi] = (1-\varepsilon^{2m+1})Q_{\lambda}[\varphi] + \varepsilon^{2m+1}Q_{\lambda}[\varphi] + \int V |\varphi|^2dx.
\end{equation}
In what follows we will find a lower bound of each term on the right side of \eqref{equ8}. For the first term, using
\eqref{equ2.1}, we have
\begin{align}\label{equ-L2-0}
(1-\varepsilon^{2m+1})\Re Q_{\lambda}[\varphi] \geq -b_m(1-\varepsilon^{2m})\lambda^{2m+1}\|\varphi\|^2_{L^2}.
\end{align}
For the second term, we use the result in \cite[Lemma 7]{G. Barbatis4} (or prove directly as \eqref{equ6}) that for all $\varepsilon>0$
\begin{equation}\label{equ85}
| Q_{\lambda}[\varphi] - Q_0[\varphi]| \leq \varepsilon Q_0[\varphi] + C\varepsilon^{-2m + 1} \lambda^{2m}\|\varphi\|^2_{L^2}.
\end{equation}
From \eqref{equ85}, it is easy to see that for all $\varepsilon>0$
\begin{equation}\label{equ86}
\varepsilon^{2m+1}\Re Q_{\lambda}[\varphi] \geq \varepsilon^{2m+1}(1-\varepsilon)Q_0[\varphi] - C\varepsilon^2\lambda^{2m}\|\varphi\|^2_{L^2}.
\end{equation}
For the third term, we claim that
\begin{equation}\label{equ9}
\left|\int V|\varphi|^2 dx\right|\leq \varepsilon \int|(-\Delta)^{\frac{m}{2}}\varphi|^2 + C\varepsilon^{-1} V_{\varepsilon^2} \int |\varphi|^2dx
\end{equation}
for all $ \varphi \in H^m(\mathbb{R}^n)$ and $0 < \varepsilon< \varepsilon_0$ for some small $\varepsilon_0$. In fact, from the proof of \cite[Theorem 2.2, Lemma 4.1]{zheng09}, we obtain
\begin{equation*}
\left\||V|^{1/2}(\nu^m  + (-\Delta)^m)^{- 1/2}\right\|_{L^2,L^2} \leq C_1 \left\|V(\nu - \Delta)^{-m}\right\|^{1/2}_{L^1,L^1},
\end{equation*}
and
\begin{equation*}
 \|V(\nu - \Delta)^{-m}\|_{L^1,L^1} \leq C_2\varepsilon^2  + V_{\varepsilon^2}\nu^{-m}
\end{equation*}
for all $\nu>0$ and $0 < \varepsilon < 1$, where $C_1, C_2$ are  constants depending only on $m,n$. Hence
\begin{equation*}
\||V|^{1/2}(\nu^m  + (-\Delta)^m)^{- 1/2}\|_{L^2,L^2} \leq C_1\sqrt{2C_2}\varepsilon
\end{equation*}
holds for $0 < \varepsilon < 1$ and $\nu \geq (\frac{V_{\varepsilon^2}}{C_2\varepsilon^2})^{1/m}$.
This inequality gives the claim \eqref{equ9} obviously.

Now replacing $\varepsilon$ by $\varepsilon^{2m+2}$ in \eqref{equ9}, we arrive at
\begin{align}\label{equ-L2-1}
\int V|\varphi|^2dx \geq -\varepsilon^{2m+2}Q_0[\varphi] - C\varepsilon^{-(2m+2)}V_{\varepsilon^{4m+4}}\|\varphi\|^2_{L^2}, \quad \varepsilon\in(0,\varepsilon_0).
\end{align}
Combining the lower bounds \eqref{equ-L2-0},\eqref{equ86},\eqref{equ-L2-1} together implies
\begin{eqnarray*}
\Re Q_{\lambda,V}[\varphi] \geq -(b_m + C\varepsilon^2)\lambda^{2m}\|\varphi\|^2_{L^2} + \varepsilon^{2m+1}(1-2\varepsilon)Q_0[\varphi] - C\varepsilon^{-(2m+2)}V_{\varepsilon^{4m+4}}\|\varphi\|^2_{L^2},
\end{eqnarray*}
for all $\varphi\in H^m(\mathbb{R}^n)$ and $\varepsilon \in (0,\varepsilon_0)$. A scaling argument yields that
\begin{eqnarray*}
\Re Q_{\lambda,V}[\varphi] \geq -(b_m + \frac{\varepsilon}{2})\lambda^{2m}\|\varphi\|^2_{L^2} + \frac{1}{2}\varepsilon^{2m+1}Q_0[\varphi] - C\varepsilon^{-(2m+2)}V_{\varepsilon^{4m+4}}\|\varphi\|^2_{L^2},
\end{eqnarray*}
for all $\varphi\in H^m(\mathbb{R}^n)$ and $\varepsilon \in (0,\varepsilon_2)$ with $\varepsilon_2:=\{\varepsilon_0, 1/2, 1/2C\}$.

On the other hand, we deduce from \eqref{equ85} and \eqref{equ9} that
\begin{eqnarray*}
\big|\Im Q_{\lambda,V}[\varphi]\big|  \leq  2\varepsilon Q_0[\varphi]+  C\varepsilon^{-2m+1} \lambda^{2m}\|\varphi\|^2_{L^2} + C\varepsilon^{-1}V_{\varepsilon^{2}}\|\varphi\|^2_{L^2}.
\end{eqnarray*}
Thus, there exists constants $C', C''$ depending only on $m,n$ such that
\begin{eqnarray*}
\big|\Im ( Q_{\lambda,V}+\iota'(\lambda, \varepsilon))[\varphi]\big|\big/ \Re( Q_{\lambda,V}+\iota'(\lambda, \varepsilon))[\varphi]
&\leq& C'\varepsilon^{-2m}
\end{eqnarray*}
for all $\varphi \in H^m(\mathbb{R}^n)$, $0 < \varepsilon \leq \varepsilon_2$. Here and below, $\iota'(\lambda, \varepsilon) = (b_m + \varepsilon)\lambda^{2m}+C''\varepsilon^{-(2m+2)}V_{\varepsilon^{4m+4}}$.

Therefore, the numerical range of $ \Theta(Q_{\lambda,V})$ is contained in the set
$$
\left\{z \in \mathbb{C}\bigg| |arg (z-\iota'(\lambda, \varepsilon))| \leq \arctan (C''\varepsilon^{-2m})\right\}.
$$
Then $ Q_{\lambda,V}$ is a  sectorial sesquilinear form. Note that $ Q_{\lambda,V}$ is densely defined and closed in $L^2(\mathbb{R}^n)$, the existence of $\widetilde{H}_{\lambda, V}$ follows from \cite[Theorem 2.1, p. 322]{Kato} directly. Moreover, the numerical range
$$
\Theta(-(\widetilde{H}_{\lambda, V}+\iota'(\lambda, \varepsilon))):=\{-(\widetilde{H}_{\lambda, V}\varphi+\iota'(\lambda, \varepsilon)\varphi, \varphi): \varphi\in H^m(\mathbb{R}^n), \|\varphi\|_{L^2}=1\}
$$
is contained in the set $\{z\in\mathbb{C}\big|\pi - \theta \leq |arg z|\leq \pi\}$,
 $\theta = \arctan (C''\varepsilon^{-2m})$. According to \cite[Theorem 3.9, p.12]{Pazy},
 we have
\begin{equation}\label{equ11}
\left\|(\mu + \widetilde{H}_{\lambda, V}+\iota'(\lambda, \varepsilon))^{-1}\right\|_{L^2,L^2} \leq \frac{C\varepsilon^{-2m}}{|\mu|}
\end{equation}
for all $\mu \in \{z \in \mathbb{C}\backslash \{0\}\big | |arg z| \leq \frac{1}{2}(\frac{3\pi}{2} - \theta)\}$, $0 < \varepsilon \leq \varepsilon_2$.
Therefore, we have for all $0 < \varepsilon \leq \varepsilon_2$ and $t>0$,
\begin{equation*}
\left\|e^{-t(\widetilde{H}_{\lambda, V}+\iota'(\lambda, \varepsilon))}\right\|_{L^2,L^2} + \left\|t(\widetilde{H}_{\lambda, V}+\iota'(\lambda, \varepsilon))e^{-t(\widetilde{H}_{\lambda, V}+\iota'(\lambda, \varepsilon))}\right\|_{L^2,L^2} \leq C\varepsilon^{-2m}.
\end{equation*}
This inequality gives the desired conclusion.
\end{proof}

We say that the semigroup $e^{-tH_{\lambda,V}}$ and $e^{-t\widetilde{H}_{\lambda, V}}$ are consistent if $e^{-tH_{\lambda,V}}f = e^{-t\widetilde{H}_{\lambda, V}}f$ is valid for any $f \in L^1(\mathbb{R}^n) \bigcap L^2(\mathbb{R}^n)$.
\begin{theorem}\label{thm3}
Semigroups $e^{-tH_{\lambda, V}}$ on $L^1(\mathbb{R}^n)$ and $e^{-t\widetilde {H}_{\lambda, V}}$ on $L^2(\mathbb{R}^n)$ are consistent.
\end{theorem}
\begin{proof}
For  $0 \leq l,k \leq \infty$, set
\begin{equation*}
V_{l,k} =
\left\{\begin{array}{cl}
V(x), \quad -l \leq V \leq k,\\
0, \qquad \textmd{other case}.
\end{array}\right.
\end{equation*}
Replacing $V$ by $V_{l,k}$ in the definition of $Q_{\lambda, V}$, we obtain a new sesquilinear form $Q^{l,k}:= Q_{\lambda,V_{l,k}}$. Since $V \in K_{2m}$, it is easy to show that $V_{l,k}$ belongs to $K_{2m}$, $Q^{l,k}$  are closed sectorial forms. Moreover, $Q^{l,k}$ are associated with closed sectorial operators $\widetilde{H}^{l,k}$, and $-\widetilde{H}^{l,k}$ generate analytic semigroups in $L^2(\mathbb{R}^n)$ satisfying the same estimates with $e^{-t\widetilde{H}_{\lambda, V}}$ in Theorem \ref{thm2}.

Since $C_0^\infty(\mathbb{R}^n)$  is a form core of $Q_{l,~k}$ and $Q_{\infty,~k}$  \cite[Theorem 4.2]{zheng09},
$$
\Im (Q^{l,k} - Q^{\infty,k})[\varphi] = 0, \quad \Re  Q^{l,k}[\varphi] \geq\Re Q^{\infty,k}[\varphi], \quad \textmd{for all } \varphi \in C_0^\infty(\mathbb{R}^n),
$$
and by Lebesgue dominated theorem
\begin{equation*}
\lim_{l\rightarrow \infty} Q^{l,k}[\varphi] = Q^{\infty,k}[\varphi], \quad \textmd{for all }  \varphi \in C_0^\infty(\mathbb{R}^n),
\end{equation*}
an application of monotone convergence theorem from above \cite[p.455]{Kato} implies that
$$
(\mu + \widetilde{H}^{l,k})^{-1}\varphi \xrightarrow[s]{L^2}(\mu + \widetilde {H}^{\infty, k})^{-1}\varphi, \quad l\rightarrow \infty
$$
for all $\varphi \in L^2(\mathbb{R}^n)$ and $ \mu > \iota'(\lambda, \varepsilon)$.

On the other hand, by Proposition \ref{proposition} below, we have
$$
(\mu + \widetilde{H}^{\infty,k})^{-1}\varphi \xrightarrow[s]{L^2}(\mu + \widetilde {H}_{\lambda,V})^{-1}\varphi, \quad k\rightarrow \infty
$$
for all $\varphi \in L^2(\mathbb{R}^n)$ and $ \mu > \iota'(\lambda, \varepsilon)$.

Thus, by Trotter approximation theorem \cite[Theorem 2.16, p.504]{Kato} we obtain
$$
e^{-t\widetilde {H}_{\lambda,V}}\varphi = \lim_{k\rightarrow \infty}\lim_{l\rightarrow \infty}e^{-t\widetilde{H}^{l,k}}\varphi, \quad\textmd{ for } t>0, \varphi \in L^2(\mathbb{R}^n).
$$

Similarly, let $H^{l,k} = P_\lambda(D) + V_{l,k}$ with domain $\mathcal {L}^{2m,1}(\mathbb{R}^n)$. Then $-H^{l,k}$ generates an analytic semigroup on $L^1(\mathbb{R}^n)$ satisfying the same estimates in Theorem \ref{thm1}. Moreover, it follows from dominated theorem that
$$
(\mu + H_{\lambda, V})^{-1}\varphi = \lim_{l,k\rightarrow \infty}(\mu + H^{l,k})^{-1}\varphi, \quad \mu > \iota(\lambda, \varepsilon), \varphi \in L^1(\mathbb{R}^n),
$$
and then, by Trotter approximation theorem  again, we have
$$
e^{-tH_{\lambda,V}}\varphi = \lim_{l,k\rightarrow \infty}e^{-tH^{l,k}}\varphi, \quad\textmd{ for } t>0, \varphi \in L^1(\mathbb{R}^n).
$$

The above considerations reduce the theorem to the consistency of $e^{-t\widetilde{H}^{l,k}}$ and $e^{-tH^{l,k}}$, and this is the case since $V_{l,k} \in L^{\infty}(\mathbb{R}^n)$, see \cite[Theorem 2]{Davies98}.
\end{proof}

In the $L^2$ part proof of Theorem \ref{thm3}, we used a convergence theorem of sectorial forms from above type to pass the limit $l\rightarrow \infty$. To pass the limit $k\rightarrow \infty$, we need a corresponding theorem from below. However, this kind of result is only available for symmetric forms in \cite{Kato}, which is not satisfied by $Q^{\infty,k}$. Fortunately, we can still prove the following result in our setting.
\begin{prop}\label{proposition}
For all $\mu > \iota'(\lambda, \varepsilon), \lambda>0, \varepsilon\in (0,\varepsilon_2),  \varphi \in  L^2(\mathbb{R}^n)$, we have $(\mu + \widetilde{H}^{\infty,k})^{-1}\varphi \xrightarrow[s]{L^2}(\mu + \widetilde {H}_{\lambda,V})^{-1}\varphi$ as $k \rightarrow \infty$.
\end{prop}
\begin{proof}
It is easy to check that
\begin{multline}\label{equ-above-1}
(Q^{\infty,k}+\mu)[(\mu + \widetilde{H}^{\infty,k})^{-1}\varphi - (\mu + \widetilde {H}_{\lambda,V})^{-1}\varphi] + (Q_{\lambda,V}-Q^{\infty,k})[(\mu + \widetilde {H}_{\lambda,V})^{-1}\varphi]\\
 = (Q_{\lambda,V}-Q^{\infty,k})[(\mu + \widetilde {H}_{\lambda,V})^{-1}\varphi, (\mu + \widetilde{H}^{\infty,k})^{-1}\varphi].
\end{multline}
Since $V\in K_{2m}$, we find $|V|\in K_{2m}$. Then by \eqref{equ85} and \eqref{equ11}, we obtain for all $\mu > \iota'(\lambda, \varepsilon), \lambda>0, \varepsilon\in (0,\varepsilon_2),  \varphi \in  L^2(\mathbb{R}^n)$,
\begin{eqnarray*}
\lefteqn{\int |V||(\mu + \widetilde{H}^{\infty,k})^{-1}\varphi|^2dx}\\ &\leq& C\bigg(Q^{\infty,k}[(\mu + \widetilde{H}^{\infty,k})^{-1}\varphi] + (1+\lambda^{2m})\|(\mu + \widetilde{H}^{\infty,k})^{-1}\varphi\|^2_{L^2}\bigg) \\
&\leq& C(\lambda, \varepsilon,\mu)\|\varphi\|^2_{L^2},
\end{eqnarray*}
where $C(\lambda, \varepsilon,\mu)$ is independent of $k$. Similarly, for another constant $ C'(\lambda, \varepsilon,\mu)$ independent of $k$
\begin{eqnarray*}
\int |V||(\mu + \widetilde {H}_{\lambda,V})^{-1}\varphi|^2dx\leq  C'(\lambda, \varepsilon,\mu)\|\varphi\|^2_{L^2}.
\end{eqnarray*}

 Thus, by H\"{o}lder inequality
\begin{eqnarray*}
\lefteqn{\left|(Q_{\lambda,V}-Q^{\infty,k})[(\mu + \widetilde {H}_{\lambda,V})^{-1}\varphi, (\mu + \widetilde{H}^{\infty,k})^{-1}\varphi]\right|}\\
&=& \left|\int (V-V_{\infty,k}) (\mu + \widetilde {H}_{\lambda,V})^{-1}\varphi \overline{(\mu + \widetilde{H}^{\infty,k})^{-1}\varphi} dx\right|\\
 &\leq& \left(\int 2|V||(\mu + \widetilde{H}^{\infty,k})^{-1}\varphi|^2dx\right)^{1/2} \cdot \left(\int 2|V||(\mu + \widetilde {H}_{\lambda,V})^{-1}\varphi dx\right)^{1/2}  \\
&\leq& 2 \sqrt{C(\lambda, \varepsilon,\mu)C'(\lambda, \varepsilon,\mu)}\|\varphi\|^2_{L^2}.
\end{eqnarray*}
Then, by dominated convergence theorem
\begin{align}\label{equ-above-2}
\lim_{k\rightarrow \infty}(Q_{\lambda,V}-Q^{\infty,k})[(\mu + \widetilde {H}_{\lambda,V})^{-1}\varphi, (\mu + \widetilde{H}^{\infty,k})^{-1}\varphi]=0,
\end{align}
and
\begin{align}\label{equ-above-3}
\lim_{k\rightarrow \infty}(Q_{\lambda,V}-Q^{\infty,k})[(\mu + \widetilde {H}_{\lambda,V})^{-1}\varphi]=0.
\end{align}

Inserting \eqref{equ-above-2} and \eqref{equ-above-3} into \eqref{equ-above-1}, noting that $\Re (Q^{\infty,k}+\mu)  \geq \mu - \iota(\lambda, \varepsilon)>0$, we get
$$
0 \leq (\mu - \iota(\lambda, \varepsilon))\left[(\mu + \widetilde{H}^{\infty,k})^{-1}\varphi - (\mu + \widetilde {H}_{\lambda,V})^{-1}\varphi dx\right] \rightarrow 0
$$
as $k \rightarrow \infty$. This implies the desired conclusion.
\end{proof}

\begin{lemma}\label{lem5}
If $V \in K_{2m}$, then the semigroup $e^{-tH_{\lambda, V}}$ on $L^1(\mathbb{R}^n)$ can be extended to a strongly continuous semigroup on $L^p(\mathbb{R}^n)(1 \leq p < \infty)$. Moreover, it holds that
\begin{equation*}
\left\|e^{-tH_{\lambda, V}}\right\|_{L^1, L^{\infty}} \leq C\varepsilon^{-(n+1)([\frac{n}{2m}]+3)}t^{-\frac{n}{2m}}\exp\left\{t(\varepsilon^{-(n+3)}V_{\varepsilon^{n+3}} +  (b_m+\varepsilon)\lambda^{2m})\right\}
\end{equation*}
where $[s]$ denotes the integer part of real number $s$, $t > 0, 0 < \varepsilon <c$.
\end{lemma}
\begin{proof}
The first statement follows from Theorem \ref{thm1}, \ref{thm2}, \ref{thm3} and duality, interpolation arguments, thus it suffices to prove the desired estimates. To this end, we need two facts. Fact one: the Sobolev embedding
$$
\mathcal {L}^{2m,1}(\mathbb{R}^n) \hookrightarrow L^p(\mathbb{R}^n)
$$
holds for $1 < p < \frac{n}{n-2m}$ if $2m<n$ and $1<p<\infty$ if $2m\geq n$. Fact two:
$$
\|(-\triangle)^m\varphi\|_{L^1}\leq C(\|H_{\lambda,V}\varphi\|_{L^1}+(V_\varepsilon+\lambda^{2m})\|\varphi\|_{L^1}), \quad \varepsilon\in (0,1/2), \varphi\in  \mathcal {L}^{2m,1}(\mathbb{R}^n).
$$
This can be proved by \eqref{equ-L1-small} and \eqref{equ6}.

Set $N=[\frac{n}{2m}]+2$, $\frac{1}{p_1} = 1-\frac{1}{N}$.  Using the two facts, we deduce from Theorem \ref{thm1} that
\begin{align*}
\|e^{-tH_{\lambda,V}}\|_{L^1, L^{p_1}}& \leq C\|(1+(-\triangle)^m)e^{-tH_{\lambda,V}}\|_{L^1, L^1}\\
&\leq C\|(H_{\lambda,V}+(V_\varepsilon+1+\lambda^{2m}))e^{-tH_{\lambda,V}}\|_{L^1, L^1}\\
&\leq\frac{C\epsilon^{-(n+1)}}{t}\exp[t(\varepsilon^{-(n+3)}V_{\varepsilon^{n+3}} +  (b_m+\varepsilon)\lambda^{2m})]
\end{align*}
for all $t>0, \lambda>0, \varepsilon\in (0,\varepsilon_1)$. By an interpolation and duality argument, if $1 \leq p < q \leq \infty$, $\frac{1}{p} - \frac{1}{q} = \frac{1}{N}$, then for all $t>0, \lambda>0, \varepsilon\in (0,\varepsilon_1)$
\begin{equation}\label{equ-L1-infty-1}
\|e^{-tH_{\lambda, V}}\|_{L^p, L^q} \leq\frac{C\epsilon^{-(n+1)}}{t}\exp[t(\varepsilon^{-(n+3)}V_{\varepsilon^{n+3}} +  (b_m+\varepsilon)\lambda^{2m})].
\end{equation}

Now, define a positive series $\{p_k\}_{0\leq k\leq N-1}$ such that
\begin{equation*}
p_0=1;\; \frac{1}{p_{i-1}} - \frac{1}{p_{i}} = \frac{1}{N}, i = 1,2,\cdots,N-2;\; p_{N-1} = N.
\end{equation*}
By the semigroup property and \eqref{equ-L1-infty-1} we obtain
\begin{align}\label{equ-L1-infty-2}
&\|e^{-NtH_{\lambda, V}}\|_{L^1, L^\infty}\leq\|e^{-tH_{\lambda, V}}\|_{N,\infty}\|e^{-tH_{\lambda, V}}\|_{p_{N-2}, p_{N-1}}\cdots\|e^{-tH_{\lambda, V}}\|_{1, p_1}\nonumber\\
&\leq \frac{C^N\epsilon^{-N(n+1)}}{t^N}\exp\{Nt(\varepsilon^{-(n+3)}V_{\varepsilon^{n+3}} +  (b_m+\varepsilon)\lambda^{2m})\}
\end{align}
for all $t>0, \lambda>0, \varepsilon\in (0,\varepsilon_1)$.

Set $t=1/N$. Then \eqref{equ-L1-infty-2} gives for all $ \lambda>0, \varepsilon\in (0,\varepsilon_1)$
\begin{eqnarray}\label{equ-L1-infty-3}
\|e^{-H_{\lambda, V}}\|_{L^1, L^\infty}\leq  C\epsilon^{-N(n+1)}\exp\{\varepsilon^{-(n+3)}V_{\varepsilon^{n+3}} +  (b_m+\varepsilon)\lambda^{2m}\}.
\end{eqnarray}

Now we estimate  $\|e^{-tH_{\lambda, V}}\|_{L^1, L^{\infty}}$ for all $t > 0$. Following \cite{zheng09}, we use a scaling argument. Let $H_{\lambda, V, t} = e^{-\lambda a \cdot (t^\frac{1}{2m}\cdot)}(-\Delta)^me^{\lambda a \cdot (t^\frac{1}{2m}\cdot)} + tV(t^{\frac{1}{2m}\cdot}), U_tg = g(t^{\frac{1}{2m}}\cdot)$, $0<t<1$. Then
\begin{equation*}
tH_{\lambda, V}\varphi = U^{-1}_tH_{\lambda, V, t}U_t\varphi, \quad \textmd{for } \varphi \in \mathcal {L}^{2m,1}(\mathbb{R}^n)
\end{equation*}
and
\begin{equation*}
e^{-t\tau H_{\lambda, V}}\varphi = U^{-1}_te^{-\tau H_{\lambda, V, t}}U_t\varphi, \quad \textmd{for } \varphi \in L^1(\mathbb{R}^n).
\end{equation*}

Similar to \eqref{equ-L1-infty-3}, we have for $ 0 <  t < 1, 0<\varepsilon<\varepsilon_1$
$$
\left\|e^{-H_{\lambda, V, t}}\right\|_{L^1, L^{\infty}}\leq C\epsilon^{-N(n+1)}\exp\{t(\varepsilon^{-(n+3)}V_{\varepsilon^{n+3}} +  (b_m+\varepsilon)\lambda^{2m})\}.
$$
Thus, for $0 < t < 1$ we obtain
\begin{align*}
\left\|e^{-tH_{\lambda, V}}\right\|_{L^1, L^{\infty}} &\leq\|U^{-1}_t\|_{L^\infty, L^\infty}\|e^{-H_{\lambda, V, t}}\|_{L^1, L^\infty}\|U_t\|_{L^1, L^1}\\
&\leq C\varepsilon^{-(n+1)N}t^{-\frac{n}{2m}}\exp\{t(\varepsilon^{-(n+3)}V_{\varepsilon^{n+3}} +  (b_m+\varepsilon)\lambda^{2m})\}.
\end{align*}
For $t > 1$, it follow from Theorem \ref{thm1} and \eqref{equ-L1-infty-3} that
\begin{align*}
\left\|e^{-tH_{\lambda, V}}\right\|_{L^1, L^\infty} &\leq \|e^{-H_{\lambda, V}}\|_{L^1, L^\infty}\|e^{-(t - 1)H_{\lambda, V}}\|_{L^1, L^1}\\
&\leq  C\varepsilon^{-(n+1)(N+1)}\exp\{t(\varepsilon^{-(n+3)}V_{\varepsilon^{n+3}} +  (b_m+\varepsilon)\lambda^{2m})\}.
\end{align*}
The two estimates imply the desired conclusion.
\end{proof}

We can now finish the proof of Theorem \ref{thm1.3}.

\emph{\textbf{The proof of Theorem \ref{thm1.3}}}.
From Lemma \ref{lem5} we know that $e^{-t((-\Delta)^m + V)}$ is bounded from $L^1(\mathbb{R}^n)$ to $L^\infty(\mathbb{R}^n)$, then \cite[Corollary A.1.2]{Simon} there exists a measurable function $K$ on $\mathbb{R}^n \times \mathbb{R}^n$ such that
\begin{equation*}
\left(e^{-t((-\Delta)^m + V)}\varphi\right)(x) = \int K(t,x,y)\varphi(y)dy, \quad \textmd{for } \varphi \in L^1(\mathbb{R}^n).
\end{equation*}
Since $e^{-tH_{\lambda,V}}$ has the kernel
\begin{equation*}
K_{\lambda, a}(t,x,y) = e^{-\lambda a \cdot x}K(t,x,y)e^{\lambda a \cdot y},
\end{equation*}
which satisfies the same bound as $\|e^{-tH_{\lambda, V}}\|_{L^1, L^{\infty}}$, we have
\begin{multline*}
|K(t,x,y)| \leq C\varepsilon^{-(n+1)([\frac{n}{2m}]+3)}t^{-\frac{n}{2m}}\\ \cdot \exp\left\{\lambda a \cdot x - \lambda a \cdot y + t(\varepsilon^{-(n+3)}V_{\varepsilon^{n+3}} +  (b_m+\varepsilon)\lambda^{2m})\right\}.
\end{multline*}
Optimizing the above estimate with respect to $ a$ and then $\lambda > 0$  yields Theorem \ref{thm1.3}.

\begin{remark}
The problem that whether or not  Theorem \ref{thm1.3} is valid for $V\in K_{2m}$ with $m\ge 1$  was originally raised by the third author after the joint work \cite{zheng09} with Yao, where they proved the weaker result that $\|e^{-tH}\|_{L^1, L^{\infty}}\leq Ct^{-\frac{n}{2m}}$.
\end{remark}

\section{Sharp heat kernel estimates for the fractional case}\label{sec3}

In this section, the goal is to prove Theorem \ref{thm1.4}. We always assume that $\alpha\in \mathbb{R}_{+}\setminus\mathbbm{N}$. Denote by $a\wedge b=\min\{a,b\}$ and $a\vee b=\max\{a,b\}$ for $a, b\in\mathbb{R}$. Let $K_0(t,x)$ be the kernel of the free heat semigroup $e^{-t(-\triangle)^\alpha}$, namely
\begin{align}\label{equ-frac-0}
K_0(t,x)=\int_{\mathbb{R}^n}e^{ix\xi}e^{-t|\xi|^{2\alpha}}d\xi, \quad t>0, x\in \mathbb{R}^n.
\end{align}
Define
\begin{align}\label{equ-frac-1}
I(t,x)£º=\frac{t}{|x|^{n+2\alpha}}\wedge t^{-\frac{n}{2\alpha}}, \quad t>0, x\in \mathbb{R}^n.
\end{align}
It is easy to see that there exists some constant $C_0=C_0(n,\alpha)$ such that
\begin{align}\label{equ-frac-2}
\frac{t}{(|x|^2+t^{\frac{1}{\alpha}})^{\frac n2+\alpha}}\leq I(t,x)\leq C_0 \frac{t}{(|x|^2+t^{\frac{1}{\alpha}})^{\frac n2+\alpha}}.
\end{align}
First of all, we recall an estimate of $K_0$, see e.g. \cite{Blumenthal}.
\begin{lemma}\label{lemma3.1}
There exists a constant $C_1=C_1(n,\alpha)$ such that
\begin{equation}\label{equ3.1}
|K_{0}(t,x)| \leq C_1I(t,x),\quad t>0, x\in \mathbb{R}^n.
\end{equation}
\end{lemma}

Next, we need the following characterization of Kato potentials, inspired by the work of  Bogdan and Jackubowski \cite{Bogdan.K}.
\begin{lemma}\label{lemma3.2}
The potential $V \in K_{2\alpha}(\mathbb{R}^n)$ if and only if  $\lim\limits_{t\rightarrow 0}K_{V}(t)=0$, where
$$
K_{V}(t)= \sup_{x\in \mathbb{R}^n}\int_{\mathbb{R}^n}J(t,x-y)|V(y)|dy,
$$
and
$$
J(t,x)=
\left\{\begin{array}{ll}
|x|^{2\alpha-n} \wedge t^{2}|x|^{-n-2\alpha}, & 0<2\alpha<n,\\
(1\vee\ln(t|x|^{-n}))\land t^{2}|x|^{-2n}, & 2\alpha=n,\\
t^{1-\frac{n}{2\alpha}}\land t^{2}|x|^{-n-2\alpha}, & 2\alpha>n.
\end{array}\right.
$$
\end{lemma}
\begin{proof}
We divide the proof into two steps.

\textbf{Step 1.} We show that $\lim\limits_{t\rightarrow 0}K_{V}(t)=0$ implies that $V \in K_{2\alpha}(\mathbb{R}^n)$. For $\mu>0$, let $R(\mu, x)$ be the integral kernel of $(\mu+(-\Delta)^\alpha)^{-1}$, namely
$$
R(\mu, x) := \int_{\mathbb{R}^n}e^{ix\xi}(\mu+|\xi|^{2\alpha})^{-1}d\xi, \quad x\in \mathbb{R}^n,
$$
which is interpreted in the sense of distributions. Since $(\mu+(-\Delta)^{2\alpha})^{-1}$ is the Laplace transform of $e^{-t(-\Delta)^{\alpha}}$, by the definition of $K_0$, we find
$$
R(\mu,x) = \int_0^\infty e^{-t\mu}K_0(t,x)dt, \quad \mu>0.
$$
We claim that for some $C_2=C_2(\alpha,n)$
\begin{align}\label{equ-frac-3}
|R(\mu,x)|\leq C_2 J(\mu^{-1},x), \quad x \in \mathbb{R}^n,\,\, \mu>0.
\end{align}
In fact, if $2\alpha\neq n$, then \eqref{equ-frac-3} follows from the same proof in Lemma 7 of \cite{Bogdan.K}. If $2\alpha= n$, by Lemma \ref{lemma3.1} and a scaling argument, then \eqref{equ-frac-3} follows from
\begin{align}\label{equ-frac-4}
\int_0^{\infty}e^{-t}I(t,x)dt\leq C_2(1\vee |\ln{|x|^n}|)\wedge |x|^{-2n}.
\end{align}
Now we prove \eqref{equ-frac-4}. By the definition of $I(t,x)$, we have
\begin{align}\label{equ-frac-5}
\int_0^{\infty}e^{-t}I(t,x)dt\leq \int_0^{|x|^n}|x|^{-2n}te^{-t}dt + \int_{|x|^n}^\infty t^{-1}e^{-t}dt.
\end{align}
On one hand,
\begin{align}\label{equ-frac-6}
\int_0^{|x|^n}|x|^{-2n}te^{-t}dt=|x|^{-2n}(1-e^{-|x|^n}-|x|^ne^{-|x|^n}) \leq 1 \wedge |x|^{-2n}.
\end{align}
On the other hand, we have both
$$
\int_{|x|^n}^\infty t^{-1}e^{-t}dt \leq |x|^{-n}e^{-|x|^n} \leq |x|^{-2n}
$$
and
$$
\int_{|x|^n}^\infty t^{-1}e^{-t}dt \leq 1 + |\int_{|x|^n}^1 t^{-1}dt|=1+|\ln |x|^n|.
$$
Thus,
\begin{align}\label{equ-frac-7}
\int_{|x|^n}^\infty t^{-1}e^{-t}dt\leq (1+|\ln |x|^n|) \wedge |x|^{-2n}.
\end{align}
Inserting \eqref{equ-frac-6} and \eqref{equ-frac-7} into \eqref{equ-frac-5} proves \eqref{equ-frac-4}. Then the claim \eqref{equ-frac-3} follows.

Thanks to the definition of $R(\mu,x)$ and \eqref{equ-frac-3}, we find
$$
\|(\mu+(-\Delta)^{\alpha})^{-1}V\|_{L^{\infty}}\leq \sup_{x\in \mathbb{R}^n}|\int{R(\mu,x-y)|V(y)|dy}| \leq C_2K_{V}(\mu^{-1}).
$$
Recall that (see \cite{zheng09}) $V \in K_{2\alpha}(\mathbb{R}^n)$ if and only if
$$
\lim\limits_{\mu\rightarrow \infty}\|(\mu+(-\Delta)^{\alpha})^{-1}V\|_{L^{\infty}}=0.
$$
Therefore, $\lim\limits_{t\rightarrow 0}K_{V}(t)=0$ implies that $V \in K_{2\alpha}(\mathbb{R}^n)$.

\textbf{Step 2.} We show that $V \in K_{2\alpha}(\mathbb{R}^n)$ implies that $\lim\limits_{t\rightarrow 0}K_{V}(t)=0$.
In the case $0<2\alpha<n$, the conclusion follows from the same argument of Corollary 12 in \cite{Bogdan.K}. The idea
to prove other cases is similar. If $2\alpha=n$ and $V\in K_{n}$, then for $t\in(0,e^{-1})$,
\begin{align*}
&\int_{\mathbb{R}^n}{((1\vee\ln\frac{t}{|x-y|^{n}})\wedge \frac{t^2}{|x-y|^{2n}})|V(y)|dy}\\
&=\int_{|x-y|^n<t}(1\vee\ln\frac{t}{|x-y|^{n}})|V(y)|dy+\int_{|x-y|^n\geq t}\frac{|V(y)|}{|x-y|^{2n}}t^{2}dy\\
&\leq 2\int_{|x-y|^n<t}{\ln|x-y|^{-n}|V(y)|dy}+\int_{|x-y|^n\geq t}\frac{|V(y)|}{|x-y|^{2n}}t^{2}dy.
\end{align*}
The first term goes to $0$ as $t\to 0$ (uniformly for $x\in \mathbb{R}^{n}$) by the definition of Kato class. To deal with
the second term, we denote by $\mathcal {C}_{x,r}=\{y\in\mathbb{R}^n: |y_{i}-x_{i}|<r, i=1,2,\cdots,n\}$ for $x\in\mathbb{R}^{n}$, $r>0$, and $K_{s}=\{a \in \mathbb{Z}^{n}: \max|a_{i}|=s\}$ for $s\in\mathbb{N}$. Put $r=t^{\frac{1}{n}}$, we derive
\begin{align*}
\int_{|x-y|^n\geq t}\frac{|V(y)|}{|x-y|^{2n}}t^{2}dy& \leq \sum_{a \in \mathbb{Z}^{n}\backslash{\{0\}}}\int_{\mathcal {C}_{x+ar,r/2}}\frac{|V(y)|}{|x-y|^{2n}}r^{2n}dy\\
& =\sum_{s=1}^{\infty}\sum_{a \in  K_s}\int_{\mathcal {C}_{x+ar,r/2}}\frac{|V(y)|}{|x-y|^{2n}}r^{2n}dy\\
&\leq \sum_{s=1}^{\infty}\sum_{a \in  K_s}(s-\frac{1}{2})^{-2n}\int_{\mathcal {C}_{x+ar,r/2}}|V(y)|dy\\
&\leq C_3\int_{\mathcal {C}_{x+ar,r/2}}|V(y)|dy,
\end{align*}
where in the last step we used that $\sum_{a \in K_s}{1}=(2s+1)^{n}-(2s-1)^n$. Since $V\in K_n$ implies that
$$
\lim_{r\rightarrow 0}\sup_{x\in \mathbb{R}^{n}}\int_{\mathcal {C}_{x+ar,r/2}}|V(y)| dy= 0,
$$
we prove the result in the case $2\alpha=n$.

Finally, if $V\in K_{2\alpha}(\mathbb{R}^n)$ with $2\alpha>n$, then $\sup_{x}\int_{|x-y|<1}|V(y)|dy<\infty$, and thus by the same argument as above we find
\begin{align*}
\sup_{x\in \mathbb{R}^{n}}\int_{\mathbb{R}^n}(1 \wedge \frac{t^{\frac{n}{2\alpha}+1}}{|x-y|^{n+2\alpha}})|V(y)|dy<\infty,
\end{align*}
which completes the proof.
\end{proof}

Like in section 2, we define a number which will appear in the kernel estimates below. Let $\omega=eC_1C_2C_4$, $C_4=2^{\alpha-1}\vee 2^{\frac{n}{2\alpha}}$. For $0<\varepsilon\ll 1$, we set
$$
V^{\varepsilon} = \sup{}\{\sigma: \sigma \in F_\varepsilon\},
$$
where
$$
F_\varepsilon = \bigg\{\sigma\leq 1: t\in (0, \sigma), \omega K_{V}(t)\leq \varepsilon \bigg\}.
$$
We are now ready to prove Theorem \ref{thm1.4}.

\emph{\textbf{The proof of Theorem \ref{thm1.4}}}.
Let $K_{0}(t,x,y)=K_{0}(t,x-y)$, defined in \eqref{equ-frac-0}. For $1\leq j\in \mathbb{N}$, define inductively
\begin{equation}\label{equ3.2}
K_{j}(t,x,y)=\int_{\mathbb{R}^n}\int_{0}^{t}K_{j-1}(t-s,x,z)V(z)K_{0}(s,z,y)dsdz.
\end{equation}
The proof will be divided into three steps.

\textbf{Step 1.} We will prove by induction that
\begin{equation}\label{equ3.3}
|K_{j}(t, x,y)|\leq C_1(\omega K_{V}(t))^{j}I(t, x-y)\quad \textmd{ for all } j\geq 0.
\end{equation}

If $j=0$, then \eqref{equ3.3} follows immediately from Lemma 3.1.

Moreover, we notice that
\begin{align*}
I(t,x)\wedge I(s,y)&\leq (2^{\alpha-1}\frac{t+s}{|x+y|^{n+2\alpha}})\wedge(2^{\frac{n}{2\alpha}}(t+s)^{-\frac{n}{2\alpha}})\leq C_4I(t+s,x+y),
\end{align*}
and thus
\begin{align}\label{equ-frac-10}
I(t,x)I(s,y)&=(I(t,x)\wedge I(s,y))(I(t,x)\vee I(s,y)) \nonumber \\
&\leq C_4I(t+s,x+y)(I(t,x)\vee I(s,y)).
\end{align}
Meanwhile, by \eqref{equ-frac-3} we see that
\begin{align}\label{equ-frac-11}
\int_{0}^{t}I(t-s,x)ds=\int_{0}^{t}I(s,x)ds\leq e\int_{0}^\infty e^{-s/t}I(s,x)ds\leq eC_2J(t,x).
\end{align}
Since for every $x\in\mathbb{R}^n$, $t\mapsto J(t,x)$ is a nondecreasing function,
it follows by \eqref{equ3.2}, the inductive assumption, \eqref{equ-frac-10}, \eqref{equ-frac-11} and the definition of $K_{V}(t)$ that
\begin{align*}
&|K_{j}(t,x,y)|\leq C_1^2(\omega K_{V}(t))^{j-1}\int_{\mathbb{R}^n}\int_{0}^{t}(I(t-s,x-z)I(s,z-y))|V(z)|dsdz \\
&\leq C_1^2C_4(\omega K_{V}(t))^{j-1}I(t,x-y)\int_{\mathbb{R}^n}\int_{0}^{t}(I(t-s,x-z)\vee I(s,z-y))|V(z)|dsdz\\
&\leq eC_1^2C_2C_4(\omega K_{V}(t))^{j-1}I(t,x-y)\int_{\mathbb{R}^n}(J(t,x-z)\vee J(t,z-y))|V(z)|dz\\
&\leq C_1(\omega K_{V}(t))^{j}I(t,x-y).
\end{align*}
Therefore the proof of \eqref{equ3.3} is completed.

\textbf{Step 2.} We will show that for any $t>0$
\begin{equation}\label{equ3.3.2}
\lim_{N\to\infty}\|e^{-tH}-\sum_{j=0}^{N}(-1)^{j}T_j(t)\|_{L^1,L^1}=0,
\end{equation}
where
$$
(T_j(t)f)(x)=\int_{\mathbb{R}^n}{K_j(t,x,y)f(y)dy} \quad\textmd{ for } f \in L^1(\mathbb{R}^n).
$$

If $V\in K_{2\alpha}(\mathbb{R}^n)$, then \cite{zheng09} for each $\varepsilon > 0$ there exists $C_\varepsilon>0$ such that
\begin{align}\label{equ-strong-1}
\|V\varphi\|_{L^1} \leq \varepsilon \|(-\Delta)^{\alpha} \varphi\|_{L^1} + C_{\varepsilon} \|\varphi\|_{L^1}
\quad \textmd{for } \varphi \in \mathcal {L}^{2\alpha,1}(\mathbb{R}^n).
\end{align}
For every $\theta\in [0, \pi)$ and $w_0>0$, set
$$
\omega_0+ \Sigma_{\theta}:=\{z\in \mathbb{C}: |\arg (z-w_0)|\leq  \theta\}.
$$
Fixed $\theta_0\in(0,\frac{\pi}{2})$. Using \eqref{equ-strong-1} with small $\varepsilon$ and Proposition \ref{lemA2}, we find for  some $\omega_0>0$
\begin{align}
\|V(\mu+(-\Delta)^{\alpha})^{-1}\|_{L^1,L^1} &\leq \varepsilon\|(-\Delta)^{\alpha}(\mu+(-\Delta)^{\alpha})^{-1}\|_{L^1,L^1}+C_{\varepsilon} \|(\mu+(-\Delta)^{\alpha})^{-1}\|_{L^1,L^1}\nonumber\\
&\leq \frac{1}{2} \quad \textmd{for } \mu\in\omega_0+ \Sigma_{\theta_0+\frac{\pi}{2}}.\label{equ3.16}
\end{align}
Recall that $H=(-\Delta)^\alpha +V$.  It follows that for all $ \mu\in\omega_0+ \Sigma_{\theta_0+\frac{\pi}{2}}$
\begin{align}\label{equ3.17}
(\mu+H)^{-1}=\sum_{j=0}^{\infty}(-1)^jR_j(\mu)
\end{align}
in the operator norm on $L^1(\mathbb{R}^n)$, where
$$
R_j(\mu)=(\mu+(-\Delta)^{\alpha})^{-1}(V(\mu+(-\Delta)^{\alpha})^{-1})^j.
$$
Applying \eqref{equ3.16} yields
\begin{align}\label{equ3.18}
\sup_{\mu\in\omega_0+ \Sigma_{\theta_0+\frac{\pi}{2}}}\|(\mu-\omega_0)R_j(\mu)\|_{L^1,L^1}\leq 1,
\end{align}
\begin{align}\label{equ3.19}
\sup_{\mu\in\omega_0+ \Sigma_{\theta_0+\frac{\pi}{2}}}\|V(\mu+H)^{-1}\|_{L^1,L^1}\leq C_5 2^{-j}.
\end{align}
Now we denote the remainder term $R_N(\mu)$ for $\mu\in\omega_0+ \Sigma_{\theta_0+\frac{\pi}{2}}$
\begin{align}
R_N(\mu)&=(\mu+H)^{-1}-\sum_{j=0}^{N}(-1)^jR_j(\mu) \nonumber \\
&=(\mu+(-\Delta)^{\alpha})^{-1}(V(\mu+(-\Delta)^{\alpha})^{-1})^{N}V(\mu+H)^{-1}\label{equ3.20}
\end{align}
Applying  \eqref{equ3.16}, \eqref{equ3.19} and  \eqref{equ3.20} we find
\begin{align}\label{equ3.21}
\sup_{\mu\in\omega_0+ \Sigma_{\theta_0+\frac{\pi}{2}}}\|(\mu-\omega_0)R_N(\mu)\|_{L^1,L^1}\leq C_5 2^{-N}.
\end{align}
It follows from \eqref{equ3.18} and the analytic representation theorem \cite[Theorem 2.6.1]{Arendt} that for every $j\in \mathbb{N}\cup\{0\}$ the integral
$$
T'_j(t)=\int_{\Gamma}{e^{\mu t}R_j(\mu)\,d\mu}
$$
exists, where $\Gamma=\Gamma_0\cup\Gamma_\pm$, $\Gamma_0=\{\mu:\mu=\omega_0+\delta e^{i\theta}, |\theta|\leq\theta_1+\frac{\pi}{2}\}$
and $\Gamma_\pm=\{\mu: \mu=\omega_0+re^{\pm i(\theta_1+\frac{\pi}{2})}, r\ge \delta\}$ ($0<\theta_1<\theta_0$, $\delta>0$).
Moreover, the Laplace transform $T_j'(t)$ is $R_j(\mu)$. Similarly, by \eqref{equ3.21} we obtain
\begin{align}\label{equ3.22}
\left\|\int_{\Gamma}e^{\mu t}R_N(\mu)d\mu\right\|_{L^1,L^1}\leq C_6 2^{-N}e^{\omega_0 t}\to 0,\,\,\,\text{as}\,\,N\to \infty.
\end{align}
It follows from \eqref{equ3.20}, \eqref{equ3.22} and the fact that $e^{-tH}=\int_{\Gamma}{e^{\mu t}(\mu+H)^{-1}d\mu}$
\begin{equation}\label{equ3.33}
\lim_{N\to\infty}\|e^{-tH}-\sum_{j=0}^{N}(-1)^{j}T'_j(t)\|_{L^1,L^1}=0,
\end{equation}
Compared to the desired estimate \eqref{equ3.3.2}, it suffices to show that $T_j(t)=T'_j(t)$.  By the uniqueness of Laplace transforms it's enough to show that the Laplace transform of each $T_j(t)$ is $R_j(\mu)$ for $\mu>\omega_0$.
\par
To this end,  we denote
\begin{equation*}
R_{j}(\mu,x,y)=\int{R_{j-1}(\mu,x,z)V(z)R_{0}(\mu,z,y)dz},\quad  j\ge 1,
\end{equation*}
where $R_{0}(\mu,x,y)=R(\mu,x-y)$. According to Lemma \ref{lemma3.2} and the fact that
$$J(\mu^{-1},x-z)\wedge J(\mu^{-1},z-y)\leq CJ(\mu^{-1},x-y),$$
we have the following
\begin{equation}\label{equ3.6}
|R_{j}(\mu,x,y)|\leq CK_V(\mu^{-1})^{j}J(\mu^{-1},x-y)
\end{equation}
 Finally, a direct computation shows
\begin{align*}
\int_{\mathbb{R}^{n}}\int_{0}^{\infty}{|e^{-t\mu}K_{j}(t,x,y)| dtdy}&\leq C_7K_{V}(\mu^{-1})^{j}\int_{\mathbb{R}^{n}}{J(\mu^{-1},x-y)dy}
\leq \frac{ C_7K_{V}(\mu^{-1})^{j}}{\mu},
\end{align*}
which implies again by Fubini's theorem that
 \begin{align*}
\int_{0}^{\infty}{e^{-t\mu}T_j(t)f(x) dt}&=\int_{0}^{\infty}e^{-t\mu}\int_{\mathbb{R}^{n}}K_j(t,x,y)f(y) dydt\\
&=\int_{\mathbb{R}^{n}}R_{j}(\mu,x,y)(x,y)f(y)dy\\
&=R_j(\mu)f(x).
\end{align*}

\textbf{Step 3.} Now, for given $\varepsilon \in (0,1)$, and $t\in(0,V^{\varepsilon})$,  we define
$$K(t,x,y)=\sum_{j=0}^{\infty}{K_{j}(t,x,y)},$$
and the associated operator
$$
T(t)f(x)=\int{K(t,x,y)f(y)dy},\quad f\in L^1(\mathbb{R}^n).
$$
It follows that
$$|K(t,x,y)|\leq \sum_{j=0}^{\infty}C_1(\omega K_{V}(t))^j I(t,x-y) \leq \frac{C_1}{1-\varepsilon}I(t,x-y),$$
and
$$
\lim_{N\to\infty}\|T(t)-\sum_{j=0}^{N}(-1)^{j}T_j(t)\|_{L^1,L^1}=0\quad\textmd{ for }0<t<V^{\varepsilon}.
$$
In view of \eqref{equ3.3.2}, we have proved that $K(t,x,y)$ coincides with the kernel of $e^{-tH}$ for $t$ small enough. For convenience, we still denote $K(t,x,y)$ the kernel of $e^{-tH}$ for any $t>0$, and by semigroup property, we can pass the estimates above to the general case. Indeed, when $t \in (V^{\varepsilon},2V^{\varepsilon})$, since
$$
K(t,x,y)=\int_{\mathbb{R}^n}{K(t/2,x,z)K(t/2,z,y)}dz,
$$
and
$$
\int_{\mathbb{R}^n}I(t;x-y)dy=C_8
$$
 independent of $t$ and $x$, it follows from \eqref{equ-frac-10} that
\begin{align*}
|K(t,x,y)| &\leq (\frac{C_{1}}{1-\varepsilon})^{2}I(t,x-y)\int_{\mathbb{R}^n}{|K(t/2,x,z)|+|K(t/2,z,y)|dz}\\
&\leq 2C_4C_8(\frac{C_1}{1-\varepsilon})^{2}I(t,x-y).
\end{align*}
By doing this inductively, we have for $t \in (2^{n-1}V^{\varepsilon},2^{n}V^{\varepsilon})$,
$$
|K(t,x,y)| \leq  \frac{1}{2C_4C_8}(\frac{2C_{1}C_{4}C_{8}}{1-\varepsilon})^{2^{n}}I(t,x-y).
$$
If we choose $\mu_{\varepsilon, V}=\frac{C}{V^{\varepsilon}}$ with some constant $C$, then  we obtain that
$$
|K(t,x,y)| \leq  Ce^{\mu_{\varepsilon, V}t}I(t,x-y),
$$
which completes the proof of Theorem \ref{thm1.4}.

\begin{remark}
Other topics on the kernel of fractional Laplacian are discussed in references. We refer the readers to \cite{Bogdan.K08} for space-time potentials and \cite{Chen} for Dirichlet boundary conditions, respectively.
\end{remark}


\section{$L^p$ estimates for fractional Schr\"{o}dinger equation}\label{sec4}

As already mentioned in the introduction, when $\alpha$ is an integer, Theorem \ref{thm1.2} can be deduced  from Theorem \ref{thm1.3} (see \cite[Theorem 5.2]{CCO}). In fact, the approach of Davies \cite[Lemma 9]{E. B. Davies95} allows us to deduce an upper Gaussian estimate for $e^{-zH}$  on the right half plane $\mathbb{C}^+$, see also the proof of Theorem \ref{thm1} in Section \ref{sec2}. Here and below,
 $$
H=(-\triangle)^\alpha + V.
$$
It then follows from the heat kernel estimates that
\begin{equation}\label{equ4.01}
\|e^{-zH}\|_{L^p,L^p}\leq C(\frac{|z|}{|\Re z|})^{n_p+\varepsilon}, ~~~  1\leq p\leq \infty,~~~z\in \mathbb{C}^+.
\end{equation}
This implies that $iH$ generates a $(1+H)^{-\alpha}$- regularized group with $\alpha>n_p$ and satisfies \eqref{1.02}. However, the argument does not work directly for the fractional Laplacian. In particular, we do not know that whether estimates \eqref{equ4.01} holds or not in the fractional case. This is an interesting question on its own.

In the sequel, we shall prove Theorem \ref{thm1.2} for integer $\alpha$ and fractional $\alpha$ in a unified manner by adapting the method in \cite{A.Jensen1, A.Jensen2} and \cite{DN}. First, we recall the amalgams of $L^q$ and $l^p$ consisting of functions such that
$$
l^{p}(L^q(\mathbb{R}^n))=\{\varphi \in L_{loc}^{q}: \sum\limits_{k}\|\varphi\|_{L^{q}(\mathcal {C}(k))}^{p}< \infty\},
$$
where ${\mathcal {C}(k)}$ is the unit cube centered at $k$, $k \in \mathbb{Z}^{n}$. More facts on $l^{p}(L^q(\mathbb{R}^n))$-spaces can be found in \cite{four}.

\begin{lemma}\label{thm4.1}
Assume that $1\leq p\leq q \leq \infty$, $\lambda>\frac{n}{2\alpha}(\frac{1}{p}-\frac{1}{q})$, and $H_{\theta}=(-\Delta)^{\alpha}+\theta V(\theta^{\frac{1}{2\alpha}}\cdot)$. Then for sufficiently large constant $M$, $(H_{\theta}+M)^{-\lambda}$ is uniformly bounded  from $L^{p}(\mathbb{R}^n)$ to $l^{p}(L^q)$ with respect to $\theta\in(0,1]$.
\end{lemma}

\begin{proof}
Let $K_{\theta}(t,x,y)$ be the integral kernel of $e^{-tH_\theta}$ and $\tilde{U}_\theta g=\theta^{\frac{n}{2\alpha}}g(\theta^{\frac{1}{2\alpha}}\cdot)$. Then we have
\begin{equation*}
e^{-\theta H_1}g = \tilde{U}_\theta^{-1}e^{- H_{\theta}}\tilde{U}_\theta g,~~g\in L^1(\mathbb{R}^n),
\end{equation*}
and $K_{\theta}(t,x,y)=\theta^{\frac{n}{2\alpha}}K_1(\theta t,\theta^{\frac{1}{2\alpha}}x,\theta^{\frac{1}{2\alpha}}y)$. We claim that
\begin{eqnarray}\label{equ-lpLq-1}
\|e^{-tH_\theta}\varphi\|_{l^p(L^q)}\leq  Ce^{Lt}(1+t^{-\frac{n}{2\alpha}(\frac{1}{p}-\frac{1}{q})})\| \varphi \|_{L^p}.
\end{eqnarray}

We divide the proof into two cases.

\textbf{Case 1: $\alpha=m$ is an integer.} It follows from Theorem \ref{thm1.3} with some given $\varepsilon>0$ and constant $c_m$ that
\begin{eqnarray}\label{equ-lpLq-2}
|K_{\theta}(t,x,y)|&\leq& C\theta^{\frac{n}{2m}}(\theta t)^{-\frac{n}{2m}}e^{L\theta t}\exp\left(-c_m\frac{|\theta^{\frac{1}{2m}}x-\theta^{\frac{1}{2m}}y|^{\frac{2m}{2m-1}}}{(\theta t)^{\frac{1}{2m-1}}}\right)\nonumber\\
&\leq& Ct^{-\frac{n}{2m}}e^{Lt}\exp\left(-c_m\frac{|x-y|^{\frac{2m}{2m-1}}}{t^{\frac{1}{2m-1}}}\right):= \Gamma(t,x-y),
\end{eqnarray}
where $C, L, c_m$ are independent of $\theta$. We can estimate the $l^{1}(L^p)$ norm of $\Gamma(t,x)$. Indeed, combine
\begin{eqnarray*}
\|\Gamma(t,\cdot)\|_{L^p(\mathcal {C}(0))}&\leq &Ct^{-\frac{n}{2m}}e^{Lt}\left(\int {\exp\left(-pc_m\frac{|x|^{\frac{2m}{2m-1}}}{t^{\frac{1}{2m-1}}}\right)dx}\right)^\frac{1}{p}
\leq Ct^{-\frac{n}{2m}(1-\frac{1}{p})}e^{Lt},
\end{eqnarray*}
and
\begin{eqnarray*}
\sum_{k\in \mathbb{Z}^{n}\setminus\{0\}}\|\Gamma(t,\cdot)\|_{L^p(\mathcal {C}(k))}&\leq& Ct^{-\frac{n}{2m}}e^{Lt}\sum_{k\in \mathbb{Z}^{n}\setminus\{0\}}e^{-c|k|^{\frac{2m}{2m-1}}/{t^{\frac{1}{2m-1}}}}
\leq Ce^{Lt}.
\end{eqnarray*}
Thus $\|\Gamma(t,\cdot)\|_{l^1(L^p)}\leq C(1+t^{-\frac{n}{2m}(1-\frac{1}{p})})e^{Lt}$. Using estimate \eqref{equ-lpLq-2} with Young inequality of $l^p(L^q)$(see \cite{A.Jensen2}), we have
\begin{eqnarray*}
\|e^{-tH_\theta}\varphi\|_{l^p(L^q)}&\leq& \|\Gamma(t,\cdot)\ast\varphi\|_{l^p(L^q)} \leq C\|\Gamma(t,\cdot)\|_{l^1(L^r)}\|\varphi\|_{L^p} \\
&\leq& Ce^{Lt}(1+t^{-\frac{n}{2\alpha}(\frac{1}{p}-\frac{1}{q})})\| \varphi \|_{L^p}, \quad \varphi\in L^p(\mathbb{R}^n),
\end{eqnarray*}
where $\frac{1}{p}+\frac{1}{r}=1+\frac{1}{q}$. This proves \eqref{equ-lpLq-1}.

\textbf{Case 2: $\alpha>0$ is not an integer.} It follows from Theorem \ref{thm1.4} that
$$
|K_{\theta}(t,x,y)|\leq Ce^{\mu_{\varepsilon, V}t}I(t,x-y).
$$
Similar to the case $1$, it suffices to show that
\begin{align}\label{equ-lpLq-3}
\|I(t,\cdot)\|_{l^{1}(L^p)} \leq C(1+t^{-\frac{n}{2\alpha}(1-\frac{1}{p})})e^{Lt}.
\end{align}
 In fact,
\begin{align*}
\sum_{k}\|I(t,\cdot)\|_{L^{p}(\mathcal {C}(k))}\leq I_1 + I_2,
\end{align*}
where
\begin{align*}
I_{1}= \sum_{k}\left(\int_{\mathcal {C}(k)\cap{\{|x|\leq t^{\frac{1}{2\alpha}}\}}}I(t,x)^{p}dx\right)^{\frac{1}{2}}\leq \left(\int_{\mathbb{R}^{n}}I(t,x)^{p}dx\right)^{\frac{1}{p}} \leq Ct^{-\frac{n}{2\alpha}(1-\frac{1}{p})},
\end{align*}
\begin{align*}
 I_{2}=\sum_{k}\left(\int_{\mathcal {C}(k)\cap{\{|x|> t^{\frac{1}{2\alpha}}\}}}I(t;x)^{p}dx\right)^{\frac{1}{2}} \leq t\sum_{k\in \mathbb{Z}^n\backslash\{0\}}{\sup_{x\in \mathcal {C}(k)\cap{\{|x|\ge t^{\frac{1}{2\alpha}}\}}}\frac{1}{|x|^{n+2\alpha}}} \leq C.
\end{align*}
This proves \eqref{equ-lpLq-3}.

Finally, according to the formula
\begin{equation}\label{equ4.3}
(H_\theta+M)^{-\lambda}=\frac{1}{\Gamma(\lambda)}\int_{0}^{\infty}t^{\lambda-1}e^{-Mt}e^{-tH_\theta}dt,
\end{equation}
 we deduce from \eqref{equ-lpLq-1} that
\begin{eqnarray*}
\|(H_\theta+M)^{-\lambda}\varphi\|_{l^p(L^q)}\leq \frac{1}{\Gamma(\lambda)}\|\varphi\|_{L^p}\int_{0}^{\infty}(t^{\lambda-\frac{n}{2\alpha}(\frac{1}{p}-\frac{1}{q})-1}+t^{\lambda-1})e^{-(M-L)t}dt.
\end{eqnarray*}
Since $\lambda>\frac{n}{2\alpha}(\frac{1}{p}-\frac{1}{q})$, the integral in the right hand side is finite if $M>L$.
\end{proof}

In Theorem \ref{thm1.2}, we need the additional assumption that $2\alpha\ge[\frac n2]+1$ in the case that $\alpha$ is not an integer. This is due to the following
\begin{lemma}\label{thm4.2}
Let $R=(H+M)^{-1}$. We assume that $\alpha>0$ if $\alpha$ is an integer and $2\alpha\ge[\frac n2]+1$ if $\alpha$ is not an integer. Then there is a positive constant C which is independent of $k\in \mathbb{Z}^n$ such that for $t\in \mathbb{R}$,
$$
\|\langle\cdot-k\rangle^{l}e^{itR}\langle\cdot-k\rangle^{-l}\|_{L^2, L^2}< C\langle t\rangle^{l},\,\,\, l=0, 1,\ldots, [\frac n2]+1.
$$
\end{lemma}
\begin{proof}
We denote by $ad^{l}(H)=[x_i,\ldots,[x_i, H]]$, here $[A, B]=AB-BA$. We claim that for $0\leq l\leq [\frac n2]+1$, $ad^{l}(R)$ is bounded on $L^{2}(\mathbb{R}^n)$. First, it is easy to see
$$[x_{i},(-\Delta)^{\alpha}]=2\alpha\partial_i (-\Delta)^{\alpha-1},$$
and
$$[x_{i},[x_{i},(-\Delta)^{\alpha}]=2\alpha(-\Delta)^{\alpha-1}+4\alpha(\alpha-1)\partial_i^{2} (-\Delta)^{\alpha-2}.$$
More generally, we have
\begin{equation}\label{equ4.4}
\left\{\begin{array}{rll}
ad^{2l-1}(H)&= \quad\sum_{j=1}^{l}C_{\alpha,j}\partial_i^{2j-1}(-\Delta)^{\alpha-l+1-j},&2l-1\leq \alpha,\\
ad^{2l}(H)& = \quad\sum_{j=0}^{l}D_{\alpha,j}\partial_i^{2j}(-\Delta)^{\alpha-l-j},&2l\leq \alpha,
\end{array}\right.
\end{equation}
where $C_{\alpha,j}$,$D_{\alpha,j}$ are constants only depending on $j,\alpha$. In view of the expression \eqref{equ4.4}, we consider the case $ad^{2l}(R)$ only, since the situation for $ad^{2l-1}(R)$ is the same.
We note that $ad^{2l}(R)$ is a linear combination of such terms:
\begin{equation}\label{equ4.41}
R~ad^{i_1}(H)~R~ad^{i_2}(H)~R~\ldots ad^{i_r}(H)~R,
\end{equation}
where $1\leq i_r\leq 2l$, $1\leq r\leq 2l$, and $\sum_{j=1}^{r}i_{j}=2l$. According to \cite[Theorem 4.2]{zheng09}, $V \in K_{2\alpha}(\mathbb{R}^{n})$ implies that for every $\varepsilon > 0$, there exists some $C_{\varepsilon}>0$,
\begin{equation*}
\left|\int_{\mathbb{R}^n}V|f|^{2}dx\right|\leq \varepsilon \|(-\Delta)^{\frac{\alpha}{2}}f\|_{L^2}^{2}+C_{\varepsilon}\|f\|_{L^2}^2.
\end{equation*}
Then for $M$ large enough, $\partial_i^{j}(-\Delta)^{\frac{\alpha-l-j}{2}}R^{\frac{1}{2}}$ is bounded in $L^{2}$, where $j=0, 1,\ldots, 2l$ and $0\leq 2l\leq [\frac{n}{2}]+1$. Indeed,
let $\varepsilon=\frac{1}{2}$ in the inequality above, and choose $M>\lambda_\frac{1}{2}$, one has
\begin{eqnarray}\label{equ4.5}
((H+M)f,f)&\geq& \frac{1}{2}\|(-\Delta)^{\frac{\alpha}{2}}f\|_2^2+(M-\lambda_{\frac{1}{2}})\|f\|_2^2\nonumber\\
&\geq& C\|(-\Delta)^{\frac{\alpha-l}{2}}f\|_2^2,~~~\,\,\,\text{for}~~~ 0\leq l\leq \alpha.
\end{eqnarray}
Combining the fact that $\partial_i^j(-\Delta)^{\frac{-j}{2}}$ is  bounded in $L^2$ for $j\ge 0$, we obtain that
\begin{equation}\label{equ4.7}
\|\partial_i^{j}(-\Delta)^{\frac{\alpha-l-j}{2}}R^{\frac{1}{2}}\|_{L^2, L^2} \leq C,~~~\,\,\, \text{for}\,\, 0\leq l\leq \alpha.
\end{equation}

Notice that when $\alpha=m$ is an integer, then $ad^{2l}(H)=0$, if $2l>2m$ and if $2l\leq 2m$, we can apply \eqref{equ4.7} to obtain that every term in \eqref{equ4.41} is bounded in $L^2$.

When $\alpha$ is not an integer, then our assumption $2\alpha\ge[\frac n2]+1$ indicates that \eqref{equ4.7} is true for $0\leq 2l\leq [\frac n2]+1$, which is what we need.
Then from the relation
$$\frac{d}{ds}(e^{-isR}x_ie^{-i(t-s)R})=-ie^{-isR}ad(R)e^{-i(t-s)R},$$
we have
\begin{equation*}
ad^{1}(e^{-itR})=-i\int_{0}^{t}e^{-isR}ad(R)e^{-i(t-s)R}ds.
\end{equation*}
Using this fact and the claim above repeatedly, we get
$$\|ad^{k}(e^{-itR})\|_{L^2, L^2}< C\langle t\rangle^{k},~~t\in \mathbb{R},$$
for $k=0, 1,\ldots, [\frac n2]+1$. Finally, according to Lemma 3.1 in \cite{A.Jensen2}, one obtains the lemma.
\end{proof}

\begin{remark}\label{rmk4.1}
(a) When $\alpha$ is not an integer, the assumption $2\alpha\ge[\frac n2]+1$ is needed in the approach above even for the free case $H=(-\Delta)^{\alpha}$, since without this condition,  we see from \eqref{equ4.4} that in $ad^{[\frac n2]+1}(R_0)$ with $R_0=(1+(-\Delta)^{\alpha})^{-1}$, there exists terms $R_0^{\frac 12}\partial_i^{j}(-\Delta)^{\frac{\alpha-l-j}{2}}R_0^{\frac 12}$ such that $\alpha<l$, which are no longer bounded in $L^2$, due to the unboundedness of the  corresponding Fourier multiplier.

(b) We mention that there is another approach by using heat kernels (see \cite{DN}). Since the kernel of $ad^{2l}(R)$ (the case $ad^{2l-1}(R)$ is the same) is $(x_i-y_i)^{2l}R(x,y)$, where $R(x,y)$ denotes the kernel of the resolvent. Then similar to Lemma \ref{thm4.1}, we can apply formula \eqref{equ4.3} with $\lambda=1$ and use Young's inequality. We see that when $\alpha$ is an integer, then $ad^l(R)$ is bounded on $L^2$ for any $l\ge0$ due to the exponential decay of the heat kernel (see Theorem \ref{thm1.3}). However, when $\alpha$ is not an integer, the polynomial decay of the kernel (see \ref{thm1.4}) requires $2\alpha>[\frac n2]+1$ to ensure that $ad^{2l}(R)$ is $L^2$ bounded for  $2l\leq [\frac n2]+1$, which is slightly worse than the method shown above.
\end{remark}

\emph{\textbf{The proof of Theorem \ref{thm1.2}.}}
Based on Lemma \ref{thm4.1} and Lemma \ref{thm4.2}, Theorem \ref{thm1.2} follows from the same arguments in \cite[Section 4]{DN}.

\section*{Acknowledgements}
We thank the referee for the careful reading and valuable comments. Duan and Zheng was supported by the National Natural Science Foundation of China No. 61671009 and 11471129, respectively. Wang was Supported by the National Natural Science Foundation of China (No. 11701535), and the Natural Science Fund of Hubei Province Grant No. 2017CFB142. Huang was supported by the Fundamental Research Funds for the Central Universities No. 2018KFYYXJJ041.

\begin{appendix}
\renewcommand{\appendixname}{Appendix\,\,}
\section{Resolvent estimates on $L^1(\mathbb{R}^n)$}\label{sec5}
We shall prove some results concerning the resolvent estimates for $(-\triangle)^\alpha$ on $L^1(\mathbb{R}^n)$ used in the paper. Let $m(\cdot): \mathbb{C}^n \mapsto \mathbb{C}$ be a bounded function. Define the Fourier multiplier operator
$$
m(D)f=\mathcal {F}^{-1}(m(\cdot)\mathcal {F}f), \,\,\,f\in \mathcal{S}(\mathbb{R}^n),
$$
where $D=(\frac{1}{i}\frac{\partial}{\partial x_1}, \frac{1}{i}\frac{\partial}{\partial x_2}, \cdots, \frac{1}{i}\frac{\partial}{\partial x_n})$,  $\mathcal{S}$ is the Schwartz class, $\mathcal {F}$ and $\mathcal {F}^{-1}$ denote the Fourier transform and its inverse transform, respectively. The restrictions in Proposition \ref{lemA1} are slightly weaker than that in \cite[Lemma 2.1]{DDY}, but the proof is the same.
\renewcommand{\appendixname}{}
\begin{prop}\label{lemA1}
Let $m\in C^L(\mathbb{R}^n\backslash \{0\})$ for some $\frac{n}{2}<L\in \mathbb{N}$. Suppose $m$ is bounded and there are constants $\sigma_1,\sigma_2>0$ such that for all $|\gamma|=L$
$$
|D^{\gamma}m(\xi)|\leq C_{\gamma} \min(|\xi|^{\sigma_1-|\gamma|},  |\xi|^{-\sigma_2-|\gamma|} ), \quad \xi\in \mathbb{R}^n\backslash \{0\}.
$$
Then for all $1\leq p\leq \infty$, $m(D)$ can be extended to a bounded operator on $L^p(\mathbb{R}^n)$, namely
$$
\|m(D)\|_{L^p,L^p}\leq C.
$$
\end{prop}

If $m(D)$ is a bounded Fourier multiplier operator on $L^p(\mathbb{R}^n)$, then its norm is invariant under scaling (see \cite[p.145]{Grafakos}), i.e.,
\begin{equation}\label{equA2}
\|m(D)\|_{L^p,L^p}=\|m(hD)\|_{L^p,L^p},\,\,\,\text{for any}\,\,h>0.
\end{equation}

\begin{prop}\label{lemA2}
Let $\alpha>0$ and $\Delta$ be the Laplacian on $\mathbb{R}^n$. Then there exists a constant $C=C(n,\alpha)$ such that
$$
\|(\mu+(-\triangle)^\alpha)^{-1}\|_{L^1,L^1}\leq
 \left\{
\begin{array}{ll}
C/|\mu|, & \mu \in \mathbb{C}\backslash \{0\},\, |\theta|\leq \frac{\pi}{2},\\
C(\sin \theta)^{-([n/2]+2)}/|\mu|, &\mu \in \mathbb{C}\backslash \{0\},\,\frac{\pi}{2}<|\theta|<\pi.
\end{array}\right.
$$
where $\theta=\arg \mu$.
\end{prop}
\begin{proof}
Thanks to \eqref{equA2}, it suffices to show that
\begin{align}\label{equA3}
\|(e^{i\theta}+(-\triangle)^\alpha)^{-1}\|_{L^1,L^1}\leq \left\{
\begin{array}{ll}
C, &  |\theta|\leq \frac{\pi}{2},\\
C(\sin \theta)^{-([n/2]+2)}, &\frac{\pi}{2}<|\theta|<\pi.
\end{array}\right.
\end{align}
The symbol of $(e^{i\theta}+(-\triangle)^\alpha)^{-1}$ is $m^{-1}$, where $m(\xi)=e^{i\theta}+|\xi|^{2\alpha}$. By induction, one can show that for all $|\gamma|>0$
\begin{align}\label{equA4}
D^\gamma m^{-1} = \sum_{1\leq j\leq |\gamma|\atop \sum |\beta|\cdot k_\beta=|\gamma|,\, \sum k_\beta=j}m^{-1-j}\prod_{0< \beta\leq \gamma}C_{k_\beta}(D^\beta m)^{k_\beta}.
\end{align}
Since if $\alpha$ is an integer, the proof of \eqref{equA3} is similar, we assume that $0<\alpha$ is not an integer. It is easy to see that
$$
|m^{-1}|\leq
\left\{
\begin{array}{ll}
\frac{2}{1+|\xi|^{2\alpha}}, & |\theta|\leq \frac{\pi}{2},\\
\frac{2}{\sin \theta(1+|\xi|^{2\alpha})}, &\frac{\pi}{2}<|\theta|<\pi,
\end{array}\right.
$$
and
$$
|D^{\beta}m|\leq C(\alpha,\beta)|\xi|^{2\alpha-|\beta|}, \quad \beta\neq 0.
$$
Let $|\gamma|>0$. Using the two estimates, we deduce from \eqref{equA4} that if $ |\theta|\leq \frac{\pi}{2}$
\begin{align}\label{equA5}
|D^\gamma m^{-1}| &\leq \sum_{1\leq j\leq |\gamma|\atop \sum_{\beta} |\beta|k_\beta=|\gamma|, \sum_\beta k_\beta=j}2^{1+j}(1+|\xi|^{2\alpha})^{-(1+j)}|\xi|^{2\alpha j-|\gamma|}\prod_{0< \beta\leq \gamma}C_{k_\beta}C^{k_\beta}(\alpha,\beta)\nonumber\\
&\leq C(n,\gamma)\min\{|\xi|^{2\alpha-|\gamma|}, |\xi|^{-2\alpha-|\gamma|}\},
\end{align}
and if $\frac{\pi}{2}<|\theta|<\pi$
\begin{align}\label{equA6}
|D^\gamma m^{-1}| \leq C(n,\gamma)\sin\theta^{-1-|\gamma|}\min\{|\xi|^{2\alpha-|\gamma|}, |\xi|^{-2\alpha-|\gamma|}\}.
\end{align}
Then \eqref{equA3} follows from \eqref{equA5}, \eqref{equA6} and Proposition \ref{lemA1}.
\end{proof}

\end{appendix}

\end{document}